\documentclass[11pt,oneside]{amsart}
\usepackage{graphicx}
\usepackage{psfrag}
\usepackage{amscd}

\usepackage{epsfig}

\usepackage{amsmath,ifthen, amsfonts, amssymb, amsopn,amsthm}
\usepackage{color}

\newcommand{\mbd}{{\mathbf{d}}}

\catcode`\@=11
\long\def\@savemarbox#1#2{\global\setbox#1\vtop{\hsize\marginparwidth
  \@parboxrestore\tiny\raggedright #2}}

\newcommand {\me}{\medskip}

\newcommand{\showcomments}{yes}

\newsavebox{\commentbox}
{\ifthenelse{\equal{\showcomments}{yes}}%
{\footnotemark
        \begin{lrbox}{\commentbox}
        \begin{minipage}[t]{1.25in}\raggedright\sffamily\tiny
        \footnotemark[\arabic{footnote}]}
{\begin{lrbox}{\commentbox}}}%
{\ifthenelse{\equal{\showcomments}{yes}}%
{\end{minipage}\end{lrbox}\marginpar{\usebox{\commentbox}}}
{\end{lrbox}}}

\newcommand{\fp}{{\mathfrak p}}
\newcommand{\fs}{{\mathfrak s}}

\newcommand{\acts}{\curvearrowright}

\newcommand{\nn}{{\mathcal N}}

\newcommand{\onn}{\overline{{\mathcal N}}}
\newcommand{\bb}{{\mathcal B}}

\newtheorem{thm}{Theorem}[section]
\newtheorem{lem}[thm]{Lemma}

\newtheorem{cor}[thm]{Corollary}

\newtheorem{prop}[thm]{Proposition}

\theoremstyle{definition}
\newtheorem{defn}[thm]{Definition}
\newtheorem{rem}[thm]{Remark}

\newtheorem{exmp}[thm]{Example}
\newtheorem{exmps}[thm]{Examples}

\newtheorem{notation}[thm]{Notation}

\newtheorem{quest}[thm]{Question}

\newcommand{\field}[1]{\mathbb{#1}}

\newcommand{\reals}{\ensuremath{\field{R}}}

\newcommand{\pdist}{{\mathrm{pdist}}}
\newcommand{\dist}{{\mathrm{dist}}}

\newcommand {\N}{\mathbb{N}} 
\newcommand {\R}{\mathbb{R}} 
\newcommand {\Q}{\mathbb{Q}} 
\newcommand {\C}{\mathbb{C}} 
\newcommand {\ww}{\mathcal{W}} 
\newcommand {\hh}{\mathcal{H}} 
\newcommand {\mm}{\mathcal{M}} 

\newcommand {\eff}{geodesically }
\newcommand {\effs}{geodesically}

\begin{document}

\title[Median geometry: spaces and groups]{Median geometry for spaces with measured walls and for groups}

\author[I.~Chatterji]{Indira Chatterji}
           \address{
Laboratoire Dieudonn\'e\\
Campus Valrose\\
06000 Nice\\
France.}
           \email{indira@unice.fr}

           \author[C.~Dru\c{t}u]{Cornelia Dru\c{t}u}
           \address{Mathematical Institute \\
 University of Oxford \\
Radcliffe Observatory Quarter\\
  Oxford, UK.}
           \email{Cornelia.Drutu@maths.ox.ac.uk}

\date{\today}

\begin{abstract}
We show that uniform lattices of isometries of products of real hyperbolic spaces act properly discontinuously and cocompactly on a median space. For lattices in products of at least two factors, this is the strongest degree of compatibility possible with median geometry. The result follows from an analysis of a quasification of median geometry that provides a geometric characterization of spaces at a finite Hausdorff distance from a median space. The case of complex hyperbolic metric spaces is different; we show that these spaces cannot be at finite Hausdorff distance from a median space. 
\end{abstract}

\maketitle

\tableofcontents

\section{Introduction.}
\textbf{1.a. Median spaces and spaces with walls.}\label{med}\quad A \emph{median} space is a metric space $(X,d)$ such that every triple of points $x_1,x_2,x_3\in X$ admits a unique \emph{median point}, i.e. a point $m\in X$ satisfying 
$$d(x_i,m)+d(m,x_j)=d(x_i,x_j), \; \hbox{for all }i,j\in \{ 1,2,3 \}, i\not=j.
$$ 
The map $X\times X\times X\to X,\; (x_1,x_2,x_3)\mapsto m,$ endows $X$ with a {\it ternary algebra structure}, called \emph{median algebra}. These were studied in \cite{VandeVel:book,BandeltHedlikova,Isbell,Sholander1,Sholander2}. The geometric study of median spaces was started in \cite{Roller:median,NicaMaster}, and more recently by Bowditch in \cite{Bow14,Bow16,Bow20}.

Examples of median metric spaces are trees, $\reals^n$ with the $\ell^1$ metric for any $n\geq 1$, and CAT(0) cube
complexes on which the Euclidean metric on cubes is replaced by the $\ell^1$ metric. According to Chepoi
and Gerasimov \cite{Chepoi:graphs,Gerasimov:semisplittings,Gerasimov:fixedpoint} the
class of $1$-skeleta of CAT(0) cube complexes coincides with the class of {\it median graphs}, which are simplicial graphs whose $0$-skeleton with the combinatorial distance is median. General median spaces can be thought of as non-discrete versions of $0$-skeleta of CAT(0) cube complexes, in the same way in which real trees are non-discrete generalizations of simplicial trees. Indeed, according to Bowditch \cite{Bow14,Bow16}, the metric of a complete connected finite rank median metric space has a bi-Lipschitz equivariant deformation that is CAT(0). \footnote{Finite rank seems to be the optimal condition for the existence of such a deformation (see Section \ref{defn:Ydelta} for the notion of rank).}

The interest of median geometry comes, among other things, from the relevance of median graphs in graph theory and computer science \cite{ChepoiBandelt}, and in optimization theory (see \cite{MMR,Wildstrom} and references therein). Moreover, two important properties of infinite groups, Kazhdan's Property (T) and a-T-menability, can be reformulated in terms of actions on median spaces \cite{CDH-adv}.
The {\it $\delta$-hyperbolicity} of metric spaces, characterized by the fact that all geodesic triangles are $\delta$-thin, is connected with both median and Banach space geometry (e.g. G. Yu \cite{Yu} showed that every hyperbolic group acts properly on an $\ell^p$-space, for $p$ large enough). In this paper, we investigate the following common generalization.
\begin{defn}\label{dmedian}[roughly median (pseudo-)metric spaces and groups]
We say that a (pseudo-)metric space $(X,\pdist)$ is \emph{roughly median} (or \emph{$\delta$--median} if we want to emphasize the constant) if there is $\delta,D\geq 0$ such that, given any three points $x,y,z$, there is a point $m\in X$
$$m\in I_{2\delta}(x,y)\cap I_{2\delta}(y,z)\cap I_{2\delta}(z,x)\subseteq B(m,D)$$ 
where a {\it rough interval} or {\it $\delta$-interval} is the set
$$I_\delta(a,b)=\{t\in X\,|\,\pdist(a,t)+\pdist(t,b)\leq\pdist(a,b)+\delta\}.$$
Namely, we require any triple intersection of $\delta$-intervals to be non-empty and of diameter at most $D$. A group is called {\it roughly median} if it admits a geometric action (i.e. properly discontinuous and cocompact) on a roughly median metric space.
\end{defn}
This natural notion encompasses both median geometry and hyperbolicity, and is an instance of coarse median geometry, but is more restrictive as, for instance, the Euclidean space ${\R}^2$ is not roughly median, but is coarsely median. Whether any coarse median space can be given a rough median metric is unclear. In \cite{NWZ} the authors have a notion of intervals that allows to express a rough median, but their relation with metric intervals is unclear. The class of roughly median groups is contained in the a priori larger class of coarse median groups, and includes mapping class groups \cite{Petyt}, see Proposition \ref{Petyt}. It is stable under direct products, relative hyperbolicity (see Proposition \ref{relY}) and rough isometries, that is quasi-isometries with multiplicative constant one.

Another structure that appears naturally in the study of actions on Banach spaces is that of \textit{space with measured walls} (see Definition \ref{defn:spacemw}). We proved in \cite{CDH-adv} that, in some sense, the category of spaces with measured walls is equivalent to the one of subspaces of pseudo-metric median spaces (see Section \ref{defn:Ydelta}), by constructing, for any given wall space $X$, a median space $\mm (X)$ containing it. 

In the present work, we show that a $\delta$-median structure on a wall space $X$ is in fact equivalent to being at finite distance from the associated median space $\mm (X)$.
\begin{thm}\label{thm:main}
Let $(X, \mathcal{W} , \mu )$ be a space with measured walls, $\mu$-locally finite (Definition \ref{defn:mulocfinite}), and such that all its half-spaces are quasi-convex (Definition \ref{quasi-conv}). The following are equivalent:
\begin{enumerate}
    \item the metric space $(X,\dist)$ associated to the wall structure is $\delta$-median;
    \item the medianization $\mm(X)$ is at finite Hausdorff distance from $X$;
    \item $X$ admits an isometric embedding $\varphi$ into a median space $\mm$ such that $\mm$ is within finite Hausdorff distance from $\varphi (X)$. 
\end{enumerate}
\end{thm}
A more precise and expanded formulation of the above can be found in Theorem \ref{thm:equiv}, and combined with Proposition \ref{pro:totallybded} allows us to deduce the following.
\begin{cor}\label{chr}
The real hyperbolic space $\field H^n$ embeds isometrically and ${\mathrm{Isom}} (\field H^n)$--equi\-va\-riantly  into a locally compact median space at finite
Hausdorff distance from the embedded~$\field H^n$.
\end{cor}
In particular, the full isometry group of the real hyperbolic space $\field H^n$, as well as all its uniform lattices, act properly and cocompactly on the locally compact median space associated to the usual measured walls structure on $\field H^n$.

\medskip

\textbf{1.b. Cubulable and medianizable groups.}\quad The previous results bring to light the existence, in the class of finitely generated groups, of several degrees of compatibility with median geometry, starting with cubulable, and proceeding in a decreasing order of strength, as follows.

\begin{defn}\label{def:cubmed}
A group is said to be
\begin{itemize}
\item {\emph{cubulable}} if it acts geometrically (i.e. properly discontinuously and cocompactly) on a CAT(0) cube complex; 
\item {\emph{strongly medianizable}} if it acts geometrically on a median space of finite rank (Definition \ref{rank});
\item {\emph{medianizable}} if it acts geometrically on a median space.
\end{itemize}
\end{defn}
Note that the cocompactness assumption is crucial here, and the study of proper actions on median spaces, according to \cite{CDH-adv}, is the one of {\it a-T-menability} (also sometimes called {\it Haagerup property}, or {\it antiT}). If a finitely generated group is cubulable, the CAT(0) cube complex on which it acts has finite dimension, so finite rank as a median space and if it is medianizable then the median space on which it acts is locally compact, without necessarily having finite rank in the sense of median spaces.

The difference between strongly medianizable versus cubulable is unclear, at this point it even seems possible that the two properties are equivalent. Some evidence comes from the fact that key properties known for cubulable groups (Tits alternative, superrigidity) have been proven for strongly medianizable groups as well \cite{F1,F3,F4}. Recent work by Messaci \cite{Me}, showing that a connected locally compact median space of finite rank which admits a transitive action is isometric to $\R^n$ endowed with the $\ell^1$-metric, provides further evidence for the equivalence of the two properties. Note that a Rips theorem for median spaces of finite rank as discussed in Section 1.d would give a weaker result than the equivalence between cubulable and strongly medianizable, if the connection between stabilizers is similar to the one established for actions on trees, as in \cite{BestvinaFeighn:stableactions,GLP,Sela:acces,Guirardel:trees,Best}. 

\medskip

On the other hand, the distinction between medianizable and strongly medianizable groups is clear: there are interesting examples of groups that are in the former class but not in the latter. Indeed, Corollary \ref{chr} implies the following. 
\begin{cor}\label{bchf}Uniform lattices in a product $SO(n_1,1)\times\dots\times SO(n_k,1)$, with $k\geq 1$, are medianizable.\end{cor}
For $k\geq 2$, irreducible uniform lattices in products $SO(n_1,1)\times\dots\times SO(n_k,1)$ are not cubulable, due to results of Chatterji-Fernos-Iozzi \cite{CFI}. Moreover, by work of Fioravanti \cite{F1,F4}, they are not strongly medianizable either and, in fact, any isometric action on a finite rank median space has a fix point. Thus, Corollary \ref{bchf} is the best one can hope for, in terms of median geometry, for these lattices. It is unknown if these same lattices can act properly on an infinite dimensional CAT(0) cube complex.

\me

Even in the case of one factor ($k=1$), Corollary \ref{bchf} may turn out to be significant. It is not known if all arithmetic uniform lattices in $SO(n,1)$, with $n$ odd and larger than $3$, are cubulable, cocompactly or even just properly. It is for instance the case for the uniform lattices described in \cite{VS,LiMillson} and \cite[$\S 6$]{Ka}. The general consensus seems to be that these lattices are cubulable, except for the construction in dimension 7 \cite{Be}, especially for the lattices constructed using octonions instead of quaternions \cite[Theorem 6.7]{Ka}. For the congruence subgroups of these latter lattices, it is proved in \cite{BC} that the first Betti number is always zero. Thus, to cubulate these lattices one would have to find finite index subgroups other than congruence subgroups, and it is not known if such subgroups exist.

\medskip

\textbf{1.c. Other rank one symmetric spaces.}\quad Quaternionic hyperbolic spaces and the Cayley hyperbolic plane over the field of octonions cannot be equipped with measured walls structures, due to the fact that their groups of isometries have Property (T). As far as the complex hyperbolic spaces are concerned, we explained in \cite{CDH-adv} how, using results of Faraut and Harzallah \cite{FarautHarz}, they can also be
equipped with structures of spaces with measured walls (the wall metric being in this case
$\dist^{\frac{1}{2}}$, where $\dist$ is the hyperbolic metric), hence their groups of isometries also act properly by isometries on a median space. However, the action is no longer cobounded. 
\begin{cor}\label{HC}
\begin{enumerate}
\item\label{nowalls} The space $(\field H^n_\C , \dist )$ cannot be isometrically embedded into a median space. In particular, the hyperbolic distance $\dist$ cannot be a wall metric. 

\medskip

\item\label{complexsnow} For any $\alpha \in [1/2,1)$, whenever $\dist^{\alpha}$ is a wall metric, $(\field H^n_\C, \dist^\alpha )$ cannot be at bounded Hausdorff distance from a median space.  
\end{enumerate}
\end{cor}

\me

\textbf{1.d. Possible further applications and open questions.}

\me

\textbf{Infinite dimensional real hyperbolic spaces.} \quad These spaces can be defined either as quotients of stabilizers of quadratic forms of signature $1$ on infinite dimensional real vector spaces \cite[$\S 6.A.III$]{Gro}, or by an infinite dimensional hyperboloid model \cite{BIM,MonodPy14}. They have a natural structure of measured walls, inducing a metric that equals the hyperbolic metric (see Remark \ref{infdim} and references therein for details).

Some important groups have interesting actions on these spaces, such as the groups of automorphisms of (products of) regular trees and their lattices (including the Burger-Mozes examples) \cite{BIM} and the groups of birational transformations of complex K\"ahler surfaces \cite{CantatAnn11}. It seems natural to ask the following.
\begin{quest}\label{ques:infdim}
Is the infinite dimensional hyperbolic space at finite Hausdorff distance from the median space associated to it?
\end{quest}

\me

\textbf{Rips-type theorems for median spaces.}\quad Our results are also relevant in the setting of potential extensions of Rips-type theorems to actions on median spaces \cite[Question 1.11]{CDH-adv}. The existing Rips-type theorems provide conditions under which a non-trivial (i.e. without global fixed point), minimal action of a group $G$ on a real tree $T \neq \R$, yields a non-trivial cocompact action on a simplicial tree, with stabilizers of edges virtually cyclic extensions of stabilizers of arcs in $T$ (see \cite{BestvinaFeighn:stableactions,GLP,Sela:acces,Guirardel:trees,Best} and references therein). 

For instance, Bestvina and Feighn prove in \cite[Theorem 9.5]{BestvinaFeighn:stableactions} that if $G$ is finitely presented, and the action $G\acts T$ is stable (a condition satisfied by free actions)  then the required action on a simplicial tree can always be produced (with virtually cyclic stabilizers of edges, if the initial action was free). 

The interest of a Rips-type theorem for median spaces is that it would relate the negation of property (T) to actions on CAT(0) cube complexes, and it would provide conditions under which a-T-menability implies weak amenability with Cowling-Haagerup constant 1 \cite[$\S 1.3$]{CDH-adv}.

The results in this paper show that \emph{one cannot expect, for actions on median spaces, a theorem similar to the one of Bestvina-Feighn}: uniform irreducible lattices in products $SO(n_1,1)\times\dots\times SO(n_k,1)$ are finitely presented, they act properly discontinuously, minimally and with compact quotient on median spaces, but cannot act non-trivially on a CAT(0) cube complex \cite{CFI} when $k\geq 2$. Still, the following question may still have a positive answer 

\begin{quest}
It is possible to obtain Rips-type theorems for stable (e.g. proper) actions of finitely presented groups on median spaces of finite rank?
\end{quest}

This is consistent with the case of real trees, which are median spaces of rank one. For such Rips theorems, the most appropriate condition corresponding to the condition $T\neq \R$ for trees seems to be ``median space with no global fixed point at infinity under the full isometry group, and which is not within bounded Hausdorff distance from a space $\R^n$ with the $\ell^1$ norm''. This follows from the fact that a finite rank median metric space has a bi-Lipschitz equivariant metric that is CAT(0), and from the results in \cite{CapraceMonod}.

\me

\textbf{1.e. Plan of the paper.} In Section \ref{defn:Ydelta} we recall the definition of (pseudo-)median space, and define and discuss other relevant notions, such as that of $\delta$-tripodic and of \eff $\delta$-tripodic spaces. Section \ref{sec:medianization} recalls the construction of a median space associated to a space with walls (this is done in more detail in \cite{CDH-adv,F1,F2}). Section \ref{sec:Hdist} is devoted to the proof of Theorem \ref{thm:equiv}. In Section \ref{RhypLocCpct} we prove that the medianization of the real hyperbolic space is locally compact. Section \ref{Chyp} discusses the complex hyperbolic case and Corollary \ref{HC}.
\medbreak

\noindent {\it Acknowledgments:} The work on the present paper started during visits to
the Universities of Paris XI (Paris-Sud) and Lille 1, while our collaboration with Fr\'ed\'eric Haglund \cite{CDH-adv} was carried out. We thank Fr\'ed\'eric for the insight he gave us on the subject, and for proofs of the implications $\eqref{(3)}\Rightarrow\eqref{convex}$ and $\eqref{convex}\Rightarrow\eqref{Haus}$ in Theorem \ref{thm:equiv} for a strengthening of $\delta$-tripodic spaces, that allowed us to find the final equivalence. We thank Elia Fioravanti for many interesting discussions on median spaces, and Thomas Delzant for explanations on complex hyperbolic spaces. We also thank Pierre-Emmanuel Caprace, Anthony Genevois, Graham Niblo and Harry Petyt for useful comments and corrections. We are also truly grateful to an anonymous referee, whose insightful questions significantly improved the paper.

The authors would like to thank MSRI and the Isaac Newton Institute for Mathematical Sciences, Cambridge, for support and hospitality during the programmes ``Geometric Group Theory'' (MSRI) and ``Non-positive curvature: group actions and cohomology'' (INI), the latter programme supported by EPSRC grant no EP/K032208/1. The second author would also like to thank Max Planck Institute for Mathematics in Bonn for its hospitality during the summers of 2021 and 2022. The first author is partially supported by IUF and ANR GAMME, and the second author by the EPSRC grant no EP/K032208/1 entitled “Geometric and analytic aspects of infinite groups” and by the LABEX CEMPI.
\section{Tripodic and median spaces.}\label{defn:Ydelta}
Recall that a \emph{pseudo-metric space} $(X,\pdist )$ is a space such that $\pdist$ satisfies all the properties of a distance, except for ``$\pdist (x,y)=0 \Rightarrow x=y$''. Its \emph{separation} is the metric space $\widehat{X}$ obtained as a quotient by the equivalence relation $\pdist (x,y)=0$.

\begin{defn}\label{rank}
A point $a$ in the pseudo-metric space $X$ is said to be \emph{between} two other points $x,y$ in $X$ if 
$$\pdist (x,a)+\pdist (a,y)=\pdist (x,y).$$ 
The \emph{interval} $I(x,y)$ of endpoints $x$ and $y$ is the set of points that are between $x$ and $y$. A point $a$ is said to be $\delta$--{\emph{between}} two other points $x,y$ if 
$$
\pdist (x,a) + \pdist (a,y)\leq \pdist (x,y)+\delta .
$$
When $\delta=0$, $a$ is between $x$ and $y$ and it belongs to the interval $I(x,y)$. The $\delta$-interval $I_\delta(x,y)$ of Definition \ref{dmedian} is the set of points that are $\delta$-between $x$ and $y$.  

A subset $A$ in $X$ is called \emph{convex} if every point $a$ that is between two points $x,y$ in $A$ lies in $A$. When $X$ is a metric space, this is equivalent to the fact that $I(x,y) \subseteq A$ for every two points $x,y$ in $A$. 

A \emph{median pseudo-metric space} is a space for which, given any triple of points $x,y,z$, the set $I(x,y) \cap I(y,z) \cap I(x,z)$ is non-empty and has diameter zero. Any point in the latter set is called \emph{median point for the triple $x,y,z$}. When $\pdist$ is a metric, this coincides with the notion of \emph{median metric space} recalled in the introduction.

The {\emph{rank}} of a median metric space $X$ is the supremum over the set of integers $k$ such that $X$ contains an isometric copy of the set of vertices $\{-a,a\}^k$ of the cube of edge length $2a$ endowed with the induced $\ell_1$-metric, for some $a>0$. Equivalently, it is the supremum over the set of integers $k$ so that there exist $k$ pairwise transverse convex walls. The {\emph{rank}} of a median pseudo-metric space $(X, \pdist )$ equals the rank of its separation.  
 \end{defn}

 

Given $R>0$, the {\em open $R$--neighbourhood} of
a subset $A$, i.e. $\{x\in X:
\pdist (x, A)<R\}$, is denoted by $\nn_R(A)$.  In particular, if $A=\{a\}$ then $\nn_R(A)=B(a,R)$
is the {\em open $R$--ball centered at $a$}. We use the notation $\onn_R (A)$ and $\bar{B}(a,R)$ to designate the
corresponding {\em closed neighborhood} and {\em closed ball},
defined by non-strict inequalities. 
\begin{rem}\label{rem:extree}
In any metric space, since for $a\in \onn_R(I(x,y))$ there is a point in $I(x,y)$ at distance less than $R$, one can see that
\begin{equation}\label{eq:nrinterval}
\onn_R(I(x,y))\subseteq I_{2R}(x,y).   \end{equation}
\end{rem}
The reverse inclusion is true in median metric spaces, but false in general. It is false in a Euclidean space, even with $\onn_R(I(x,y))$ replaced by $\onn_D(I(x,y))$, for some uniform $D=D(R)$. The following examples show spaces that are at bounded distance from a median space where the reverse inequality is still false.
\begin{exmps}\label{trou}
Consider a simplicial tree $T$, the median space $M=T\times [0,1]$, and the subspace $X$ of $M$ obtained by removing relative interiors of rectangles $[x_n, y_n]\times [0,1]$ and, in the copy $T\times \{1\}$, the relative interiors of geodesics $[x_n,y_n ]\times \{1\}$, with $[x_n,y_n ]$ pairwise disjoint geodesics in $T$ of respective length $n$. In the space $X$ with the metric induced from $M$, $I((x_n,1),(y_n,1))=\{ (x_n,1),(y_n,1)\}$, while $I_2 ((x_n,1),(y_n,1))$ contains the entire geodesic $[x_n,y_n ]\times \{0\}$ (see Figures \ref{fig:trou}).

Another example would be to consider ${\mathbb R}^2$ with the $L^1$ norm, remove the interior of pairiwise disjoint translates of squares $[0,n]\times[0,n], n\in \N$, and replace them with the other faces of a height one parallelepiped over that square, with the interior of the bottom square removed.  
\begin{figure}
    \centering
    \includegraphics[width=0.5\textwidth]{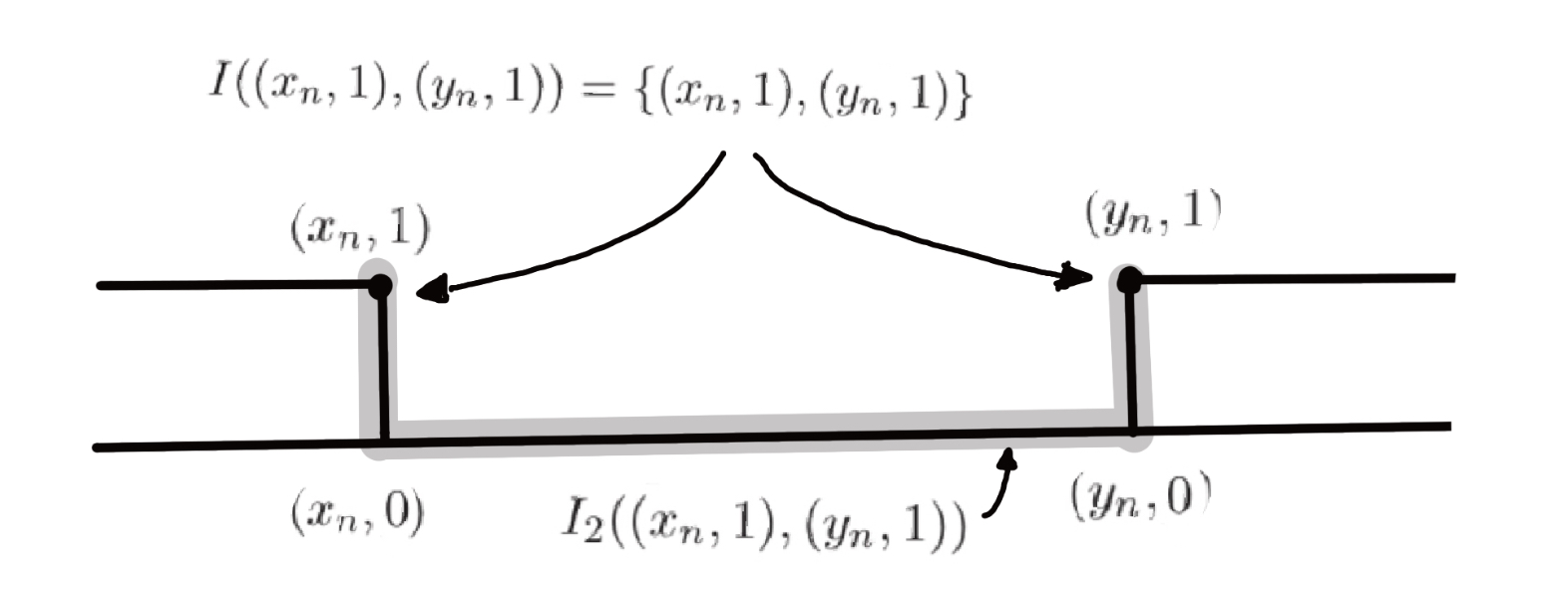}
    \caption{A line $T\times[0,1]$ with pieces of a rectangle removed}
    \label{fig:trou}
\end{figure}
\end{exmps}
\begin{rem}\label{ProdInterval}
    Let $(X_1,d_1)$ and $(X_2,d_2)$ be metric space, and $\delta_1,\delta_2\geq 0$. Then, for any $x=(x_1,x_2), y=(y_1,y_2)\in (X_1\times X_2, d_1+d_2)$, a direct computation shows that
$$I_{\delta_1}(x_1,y_1)\times I_{\delta_2}(x_2,y_2)\subseteq I_{\delta_1+\delta_2}(x,y)\subseteq I_{\delta_1+\delta_2}(x_1,y_1)\times I_{\delta_1+\delta_2}(x_2,y_2).$$
The second inclusion is false in general for the $\ell^2$ combination of the metrics, even with different constants. For instance, in the Euclidean plane $\R^2$, taking $x=(0,0)$ and $y=(2R,0)$, one checks that $t=(R,\sqrt{R\delta})\in I_{\delta}(x,y)$ is not contained in the rectangle $I_{\delta}(0,2R)\times I_{\delta}(0,0)$, when $R$ is large enough.
\end{rem}
The following generalizations of both median spaces and Gromov hyperbolic spaces requires, loosely speaking, that the set of `quasi-centers' of any triple of points is non-empty, or that the `quasi-centers' come from thin geodesic triangles.
\begin{defn}\label{Yd}[(\eff) tripodic (pseudo-)metric spaces and groups]
A (pseudo-)metric space $(X,\pdist )$ is called \emph{tripodic} (or {\it $\delta$-tripodic} when we want to specify the constant) if, given any three points $x,y,z$, 
$$I_{2\delta}(x,y)\cap I_{2\delta}(y,z)\cap I_{2\delta}(z,x)\not=\emptyset.$$

 We say that $X$ is \emph{\eff tripodic (or $\delta$-tripodic)} if given any three points $x,y,z$, there exists a $\delta$-\emph{thin} geodesic triangle with vertices $x,y,z$, namely a geodesic triangle such that each of its sides is contained in the $\delta$-neighborhood of the union of the other two sides.

A finitely generated group $G$ is {\it (\effs) tripodic (or $\delta$-tripodic)} if it admits a geometric action on a (\effs) tripodic metric space.
\end{defn}
\begin{rem}
Roughly median spaces and groups (Definition \ref{dmedian}) are in particular tripodic. A tripodic space is then roughly median if the diameter of those non-empty triple intersections is uniformly bounded. Similarly, we can talk about geodesically roughly median spaces, when one can actually see those thin triangles in the space, which is not always the case, as Example \ref{trou}
shows.\end{rem}
\begin{exmps}\label{ex:tripodic} 
\begin{enumerate}
\item In metric graph theory, $0$-tripodic graphs are called modular graphs. Thus, all the examples of modular graphs are also examples of $0$-tripodic spaces.
\item Pseudo-modular graphs (in particular Helly graphs) are $1$-tripodic \cite[$\S 2.2.2.2$]{CCHO}. 
\item Quasi-median graphs of finite cubical dimension $d$ are $d$-tripodic by work of Genevois \cite[Proposition 2.84]{Gen17}.
\end{enumerate}
\end{exmps}

\begin{rem}\label{rem:nbhd}
 Geodesically $\delta$-tripodic implies tripodic, but the converse implication is not true in general. Examples \ref{trou} depict spaces that are $1$-tripodic, but not geodesically $1$-tripodic.
These examples display the key difference between the two notions: in a $\delta$--tripodic space in general we may not have that $I_\delta(x,y)$ is a subset of $\nn_{D}(I(x,y))$ for all $x,y\in X$, for some uniform constant $D$; in a \eff $\delta$-tripodic space, on the other hand, for any $\alpha\geq 0$ there exists $D=D(\delta,\alpha)\geq 0$ such that  $I_\alpha(x,y)\subseteq\nn_{D}(I(x,y))$.
Indeed, take $t\in I_\alpha(x,y)$ and consider a $\delta$-thin triangle with vertices $x,y,t$. Take $p\in I(x,y)$, at distance less than $\delta$ from both $p_x\in I(x,t)$ and $p_y\in I(y,t)$. Then $\pdist(x,p)\leq\pdist(x,p_x)+\delta$ and $\pdist(y,p)\leq \pdist(y,p_y)+\delta$ and from $t\in I_\alpha(x,y)$ we get that
$$\pdist(x,p_x)+\pdist(p_x,t)+\pdist(t,p_y)+\pdist(p_y,y)\leq \pdist(x,p)+\pdist(p,y)+\alpha$$
hence $\pdist(p_x,t)+\pdist(t,p_y)\leq 2\delta+\alpha$ and so $2\pdist(t,p_y)-2\delta\leq 2\delta+\alpha$ since $\pdist(p_x,t)\leq \pdist(p_y,t)+2\delta$. Then $\pdist(t, p)\leq \pdist(t,p_y)+\delta\leq D(\alpha,\delta)$, so that $t\in\nn_{D}(I(x,y))$.
\end{rem}
\begin{exmps}\label{hypY} 
\begin{enumerate}
\item Geodesic Gromov hyperbolic spaces are (\effs) $\delta$-tripodic;
\item (geodesic) median spaces are (\effs) $0$-tripodic; 
\item if $(X_1, d_1)$ and $(X_2,d_2)$ are $\delta_1$ and $\delta_2$-tripodic respectively, then, from the inclusions in Remark \ref{ProdInterval}, one deduces that the product $(X_1\times X_2, d_1+d_2)$ is $(\delta_1+\delta_2)$-tripodic.  In particular, products of hyperbolic spaces are $\delta$-tripodic.
\item Euclidean (and Hilbert) spaces cannot be $\delta$-tripodic for any $\delta<\infty$ unless if they are of real dimension one. 
\item According to Huang and Osajda \cite{HO21}, certain Artin groups and Garside groups (e.g. braid groups) are tripodic.
\end{enumerate}
\end{exmps}
With the above definition and $D=0$, a 0--median space is a median space. It is immediate from the definitions that a roughly median space is coarse median in the sense of Bowditch \cite{Bow13}. 
\begin{rem}
The property of being roughly median is strictly stronger than that of tripodic: the building associated to $Sp(4,\Q_p)$ admits a metric that is $0$-tripodic. Indeed, this metric corresponds to a consistent choice of a pair of orthogonal roots in each flat. This is possible because it has two pairs of roots of different lengths. For such a metric, intervals grow exponentially, medians are not unique and can have arbitrarily large diameters. The metric cannot be median because $Sp(4,\Q_p)$ has property (T). A similar argument works for $Sp(2n, \Q_p)$ and $SL_4(\Q_p)$. They cannot be coarse median either \cite{Hae}, hence in particular not roughly median.
\end{rem}
Being tripodic or roughly median behaves well under rough isometries, which are quasi-isometries with multiplicative constant 1, see in Arnt's PhD thesis \cite{ArntPhD}. Those notions also behave well under relative hyperbolicity, as the following stability result shows.
\begin{prop}\label{relY}
Let $G$ be a finitely generated group hyperbolic relative to a finite family of finitely generated subgroups $H_1,..., H_n$. Assume that for every $i\in \{1,...,n\}$ the subgroup $H_i$ is $\delta$-tripodic (resp $\delta$-median). Then the group $G$ is $\alpha$-tripodic (resp. roughly median) for some $\alpha\geq 0$.
\end{prop}
In particular, any lattice in $SO(n,1)$ is $\delta$-tripodic. However, non-uniform lattices in $SU(3,1)$ are not $\delta$-tripodic: a group that is $\delta$-tripodic must have sub-cubic Dehn function \cite{E}, and these lattices have cubic Dehn functions. Note that M.~Elder's proof in \cite{E} is purely metric, thus even though the result is formulated for word metrics and Dehn functions, it is also true for metrics quasi-isometric to word metrics and for metric generalisations of Dehn functions, as defined in \cite[5.F]{Gro}.   
\begin{proof}[Proof of Proposition \ref{relY}]
Assume that for every $i\in \{1,...,n\}$ the subgroup $H_i$ admits a geometric action on a (\effs) $\delta$-tripodic metric space $X_i$ (resp. roughly median). According to Section 4 of \cite{Groves} or \cite{Kar}, there is a metric space $Y$ on which $G$ acts geometrically and such that the $X_i$'s are isometrically embedded. Let us show that this metric space $Y$ is (\effs) $\alpha$-tripodic (res. roughly median) if the subspaces $X_i$ are. According to \cite{DrutuSapir:TreeGraded,Drutu}, relative hyperbolicity implies that there is a constant $\delta$ such that for any triangle $\Delta$, consisting of a triple of points $x,y,z$ in $G$ and (discrete) geodesics between those points, there exists a coset $gH_i$ such that if we denote by $x_1,x_2,y_1,y_2,z_1,z_2\in gH$ the entrance and exit points of the geodesics of $\Delta$ from this coset, we have that $d(x_1,x_2),d(y_1,y_2),d(z_1,z_2)\leq\delta$. Now, since $X_i$ is (\effs) $\delta_i$-tripodic (res. roughly median) and isometrically embedded in $Y$, there is a $\delta_i$-centre for the triple $x_1,y_1,z_1$, which will work for the triple $x,y,z$ as well.
\end{proof}
The geodesical versions of $\delta$-tripodic (resp $\delta$-median) are a little more subtle as one needs to control where the actual geodesics travel, and are left to an interested reader. It is straightforward to check that products of geodesic hyperbolic spaces are roughly geodesically median. There has been a lot of lively research recently around the question of the existence of an (equivariant) embedding of a group into a finite product of hyperbolic spaces. This started with the work of Bestvina, Bromberg, Fujiwara \cite{BBF-IHES}, and later Bestvina, Bromberg, Fujiwara and Sisto \cite{BBFS}, who introduced a set of axioms allowing to construct such an embedding, and applied it, for instance, to mapping class groups. Hagen and Petyt later proved that, at least in the case of mapping class groups, this embedding endows the group with a structure of $\delta$-median space \cite{HP21}. After that, Bestvina, Bromberg and Fujiwara showed that in certain cases (e.g. for mapping class groups) the construction can be adapted to yield an equivariant quasi-isometric embedding into a finite product of quasi-trees. This construction has been proven to induce a roughly median structure on mapping class groups by Petyt \cite{Petyt} (but it is unclear if it can be obtained using word metrics). 
\begin{prop}[Petyt]\label{Petyt}
    Mapping class groups are roughly median.
\end{prop}
\begin{proof}
    Let $G$ be a mapping class group, that acts on a product of quasi-trees as provided by \cite{BBF}. According to Proposition 3.2 of \cite{Petyt}, under the orbit map, any hierarchy path in the finite index, colour-preserving subgroup of $G$, turns into a quasi-geodesic in each of the quasi-trees. It remains to see that the orbit is coarsely median-preserving. This is the content of the first two paragraphs of the proof of Theorem A, just below the statement of Proposition 3.2 on p.12, which shows that the median of 3 points in $G$ in lies on a triangle of hierarchy paths, which turns into triangles of quasi-geodesics in the quasi-trees.
\end{proof}
Alternatively, one can also use Proposition 3.9 of \cite{HP21}, which embeds the mapping class group in a product of hyperbolic spaces with coarsely median-preserving orbits (with the same argument).
\section{The median space associated to a space with measured walls.}\label{sec:medianization}
From \cite{HaPa}, we recall that a \emph{wall} of a set $X$ is a partition $X=h\sqcup h^c$ (where {$h$ is possibly empty or the whole $X$}).  A collection $\hh$ of subsets of $X$
is called a \emph{collection of half-spaces} if for every $h\in \hh$ the complementary subset $h^c$ is
also in $\hh$. We call \emph{collection of walls} on $X$ the collection $\ww_\hh$ of pairs $w=\{h,h^c
\}$ with $h\in \hh$. For a wall $w=\{h,h^c\}$ we call $h$ and $h^c$  {\em the two half-spaces bounding
$w$}. Another wall $w'=\{h',{h'}^{c}\}$ is said to \emph{intersect} $w=\{h,h^c\}$ if each half-space bounding $w$ intersects non-trivially each half-space bounding $w'$, and the walls $w,w'$ are \emph{parallel} otherwise. We say that a wall $w=\{h,h^c \}$ {\em separates} two disjoint subsets $A,B$ in $X$ if $A\subseteq h$ and
$B\subseteq h^c$ or vice-versa and denote by $\ww (A| B)$ the set of walls separating $A$ and $B$. In
particular $\ww (A |\emptyset)$ is the set of walls $w=\{h,h^c \}$ such that $A\subseteq h$ or $A\subseteq
h^c$; hence $\ww (\emptyset | \emptyset )=\ww$. We use the notation $\ww (x|y)$ to designate $\ww (\{x\}|\{y\})$.
\begin{defn}[Space with measured walls \cite{CherixMartinValette}]\label{defn:spacemw}
A \emph{space with measured walls} is a set $X$, with $\mathcal{W}$ a collection of walls, $\mathcal{B}$ a $\sigma$-algebra of subsets of $\mathcal{W}$
and $\mu$ a measure on $\mathcal{B}$, such that for every two points $x,y\in X$ the set of
separating walls $\ww (x| y)$ is in $\bb$ and has finite measure. We denote by $\pdist_\mu$ the
pseudo-metric on $X$ defined by 
$$\pdist_\mu (x,y)=\mu \left( \ww (x| y) \right),$$ 
and we call it the
\emph{wall pseudo-metric}.
\end{defn}
\begin{rem}
Consider the set $\hh$ of half-spaces determined by $ \ww$, and the natural projection map $\fp :\hh
\to \ww$, $h\mapsto \{h,h^c \}$. The pre-images of the sets in $\mathcal{B}$ define a $\sigma$-algebra
on $\hh$, {which we still (abusing notations) denote by $\bb$}; hence on $\hh$ can be defined a pull-back measure that we also denote by $\mu$. This allows us to work either in $\hh$ or in $\mathcal{W}$.
\end{rem}
For the definition of homomorphisms of spaces with measured walls we use a slightly modified terminology, in accordance with the one in \cite{F1,F2}.
\begin{defn}\label{defn:hommesw}
Let $(X,\ww , \mu)$ and $(X',\ww' ,\mu')$ be two spaces with measured walls. A map $\phi:X\to X'$ is a \emph{homomorphism of spaces with measured walls} if:
\begin{itemize}
\item for any $w'=\{h',h'^c\}\in\ww'$ we have $\{\phi^{-1}(h'),\phi^{-1}(h'^c)\}\in\ww$;
this latter wall we denote by $\phi^*(w')$;
\item the map $\phi^* : \ww' \to \ww$ is measurable and $(\phi^*)_*\mu'=\mu$.
\end{itemize}
We say that $\phi$ is a \emph{monomorphism of spaces with measured walls} if $\phi^*$ is surjective, and that $\phi$ is a \emph{coarsely surjective monomorphism} if moreover there exists $D>0$ such that $X'$ is contained in the closed $D$-neighbourhood of $\phi (X)$, and for every half-space $h$ of $X$ there exists a half-space $h'$ of $X'$ at Hausdorff distance less than $D$ of $\phi (h)$. Both neighbourhoods are considered with respect to the wall pseudo-metric $\pdist_{\mu'}$.
\end{defn}
Note that being a monomorphism does not necessarily imply injectivity. On the other hand, being a homomorphism already does imply that the map is an isometric embedding with respect to  the wall pseudo-distances.
\begin{exmp}\label{euclidean}The plane $\reals^2$ with the Euclidean distance can be isometrically embedded in a median space, \textit{via} the map
\begin{eqnarray*}\reals^2&\to & L^1([0,2\pi])\\
\left(\begin{array}{c}x\\y\end{array}\right)&\mapsto& {\frac{1} {4}}\left(x\sin t+y\cos t\right).
\end{eqnarray*}
\end{exmp}
\begin{exmp}\label{ex:measwallsLp}
As mentioned in the introduction, a median metric space is endowed with a structure of convex measured walls, such that the wall pseudo-metric $\pdist_\mu$ coincides with the median metric, and isometries are automorphisms of the space with measured walls \cite[Section 5]{CDH-adv}. Thus, every $L^1$-space has a measured walls space structure, and every $L^p$-space, with $p\in (1,2]$ (in particular, every Hilbert space),  has a measured walls space structure, since they all embed isometrically in an $L^1$-space \cite{WellsWilliams:Embeddings}.  
\end{exmp}
\begin{exmp}[Finite dimensional real hyperbolic space]\label{exmp:realhyper} 
Define the half-spaces of the  real hyperbolic space $\field H^n$ to be closed or open geometric
half-spaces, so that the boundary of half-spaces is an isometric copy of $\field H^{n-1}$ (a geometric
hyperplane of $\field H^n$). More precisely, if we use the model of one sheet of a hyperboloid, 
$${\mathcal H}^n=\{ (x_1, \dots , x_{n+1}) \in \R^{n+1} \mid x_{n+1}^2 - \sum_{i=1}^n x_i^2 =1, x_{n+1}>0 \}$$ 
for $\field H^n$, then a closed half-space $h$ is defined by $x_1\geq 0$, with boundary hyperplane $\partial h$ defined by the equation $x_1=0$. The connected component of the identity, $SO_I(n,1)$, of the group stabilizing the form $x_{n+1}^2 - \sum_{i=1}^n x_i^2$, acts transitively on the set of unit tangent vectors of  $\field H^n$ identified with ${\mathcal H}^n$, hence all closed half-spaces are in the same orbit. The stabilizer of the half-space $h$ can be identified with the stabilizer of the vector $(1,0,\dots ,0)^t$ in $SO_I(n,1)$, thus with $SO_I(n-1,1)$ (identified with the subgroup of $SO_I(n,1)$ with $1$ in the upper left corner and all the other entries of the first row and the first column zero). Thus, we can identify the set of closed half-spaces with $SO(n,1)/SO(n-1,1)$.  

An element in the stabilizer of the hyperplane $x_1=0$ can either  fix $h$ or swap it with the opposite closed half-space, as it contains for instance the symmetry represented by the diagonal matrix $\sigma $ with diagonal $(-1, -1, 1,\dots ,1)$. Thus, the stabilizer of the hyperplane $\partial h$ is the semidirect product of $SO_I(n-1,1)$ with $\mathbb Z_2 = \{ {\mathrm id}, \sigma \}$. The set of hyperplanes can therefore be identified with $SO(n,1)/[SO(n-1,1)\rtimes \mathbb Z_2]$. 

The set of walls is of the form $\{ h, h^c \}$, with $h$ a closed half-space and $h^c$ the complementary open half-space, and can therefore be identified with the set of closed half-spaces, hence with $SO(n,1)/SO(n-1,1)$. Since the stabilizer of a wall/a closed half-space is unimodular, there is a $SO(n,1)$-invariant borelian measure $\mu_{\field H^n}$ on the set of walls. The set
of walls separating two points has a compact closure, therefore it has finite measure and thus $(\field
H^n,\ww_{\field H^n},\mu_{\field H^n})$ is a space with measured walls. By Crofton's formula \cite[Proposition 3]{CherixMartinValette}, up to multiplying the measure $\mu_{\field H^n}$ by some
positive constant, the wall pseudo-metric on $\field H^n$ coincides with the usual hyperbolic distance. The case of the infinite dimensional hyperbolic space seems more complicated, see Remark \ref{infdim}.\end{exmp}
We now recall how a space with measured walls naturally embeds in a median space. This is done in detail in \cite{CDH-adv,F1,F2} and we just give the outline here.
\begin{defn}Given a space with measured walls $(X,\ww , \mu)$, a section $\mathfrak{s}$ for the projection $\fp :\hh\to \ww$ is called \emph{admissible} if its image contains, together with a half-space $h$, all the half-spaces $h'$ containing $h$. We denote by $\overline{\mathcal M}(X)$ the set of admissible sections.
\end{defn}
\begin{rem}We identify an admissible section  $\fs$ with its image $\sigma=\fs (\ww )$; with this identification, an admissible section becomes a collection of half-spaces, $\sigma$, such that:
\begin{itemize}
    \item  for every wall $w = \{h,h^c\}$ either $h$ or $h^c$ is in $\sigma $, but never both;
    \item  if $h\subset h'$ and $h\in \sigma$ then $h'\in \sigma$.
\end{itemize}
Such a collection of half-spaces is commonly called {\it ultrafilter} in the literature on median spaces.\end{rem}
For any $x\in X$ we denote by $\sigma_x$ the image of the section of $\fp$ associating to each wall $\{h,h^c\}$ the half-space containing $x$. That is, $\sigma_x$ is the set of half-spaces $h\in\hh$ such that $x\in h$. This is an admissible section and $$\fp(\sigma_x\vartriangle\sigma_y)=\ww (x| y).$$

Let now $x_0$ denote some base point in $X$. We define 
$${\bb_0}:=\{A\subseteq\hh\hbox{ such that }A\vartriangle \sigma_{x_0}\in\bb\hbox{ and }\mu(A\vartriangle \sigma_{x_0})<+\infty\},$$
and endowed with the following map
$$\pdist_\mu(A,B)=\mu(A\vartriangle B),$$ 
the set ${\bb_0}$ becomes a median pseudo-metric space. Indeed, the identity $A\vartriangle \sigma_{x_1}=(A\vartriangle \sigma_{x_0})\vartriangle
(\sigma_{x_0}\vartriangle \sigma_{x_1})$ proves the triangle inequality, and the fact that $\sigma_{x_0}\vartriangle \sigma_{x_1}$ is
measurable with finite measure shows that the median pseudo-metric space ${\bb_0}$ is independent of the chosen base point $x_0$. In particular, $\sigma_x\in {\bb_0}$ for any $x\in X$ and the elements in $\overline{\mathcal M}(X)\setminus{\bb_0}$ are the elements at infinite distance from the sections coming from the elements of $X$, and form the Roller boundary (that we will not discuss here). The map
\begin{equation}\label{chixo}
\chi^{x_0}:{\bb_0}\to {\mathcal S}^1(\hh,\mu),\, \,  \chi^{x_0}
(A)=\chi_{A\vartriangle \sigma_{x_0}}
\end{equation}
 is an isometric embedding of ${\bb_0}$ into the median subspace
${\mathcal S}^1(\hh,\mu)\subseteq {\mathcal L}^1(\hh,\mu)$, where ${\mathcal S}^1(\hh ,\mu)=\{\chi_B\,|$
$B$ is measurable and $\mu(B)<+\infty\}$. Notice that for $x,y\in X$ we have $\pdist_\mu(x,y)=\mu(\sigma_x\vartriangle\sigma_y)$, thus $x\mapsto
\sigma_x$ is an isometric embedding of $X$ into $({\bb_0},\pdist_\mu)$.
\begin{defn}\label{def:mx}
The {\it median space associated to }$X$ is the metric space ${\mathcal M}(X)$ obtained as the separation of the pseudo-metric space 
$$\overline{\mathcal M}(X)\cap {\bb_0}.$$
Since each admissible section $\sigma_x$ belongs to ${\mathcal M}(X)$, it follows that $X$ isometrically
embeds in ${\mathcal M}(X)$. We will denote by $\iota:X\to\mm(X)$ this isometric embedding. Given two elements $\tau,\tau'$ in ${\mathcal M}(X)$, we denote 
$$\ww(\tau | \tau')=\{w=\{ h,h^c\} \in\ww\,|\, h\in\tau, h^c\in\tau'\}=\fp(\tau\vartriangle\tau').$$
\end{defn}
The following result emphasizes that the construction of ${\mathcal M}(X)$ is the right one.
\begin{prop}[Proposition 3.14, \cite{CDH-adv}]\label{prop:medianization}
The space ${\mathcal M}(X)$ is a median subspace of $\widehat{{\bb}_0}$. Let $(X',\ww')$ be another space with measured walls. Any homomorphism of spaces with measured walls $\phi :X \to X'$
induces an isometric embedding $\mm(X)\to\mm(X')$. In particular the group of automorphisms of $(X,\ww)$
acts by isometries on ${\mathcal M}(X)$.
\end{prop}
\begin{rem}\label{rem:espaceM0}
To all intents and purposes, the space $\mm (X)$ can be replaced by $\mm_0 (X)$, the metric completion of the median closure of $X$ in  $\mm (X)$. The space $\mm_0 (X)$ may in general be different from  $\mm (X)$, but is equal to it when the space $\mm_0 (X)$ is locally convex \cite{F2}.  

All the results in \cite{CDH-adv} and in this paper that are formulated for $\mm (X)$ also hold for $\mm_0 (X)$. Moreover $\mm_0 (X)$ has the advantage of being a complete geodesic metric median space when $X$ is connected. Indeed, according to Bowditch \cite{Bow16}, a complete median space is geodesic if and only if it is connected. The median map is $1$--Lipschitz, therefore the image of $X$ by it in $\mm (X)$ is connected, and as the median completion of $X$ in  $\mm (X)$ equals the increasing union of (connected) sets obtained by iterative applications of the median map to $X$, it is itself connected. The space $\mm_0 (X)$ is the metric completion of this latter median completion, hence it is itself connected.
\end{rem}
%
\section{Distances to the associated median space.}\label{sec:Hdist}
\begin{notation}\label{Nwallssubsets}
{Let $(X,\ww , \mu )$ be a space with measured walls. For any subset $Y\subseteq X$, we denote by
${\mathcal W}(Y)$ the set of walls separating two points of $Y$ (we also say that these walls \emph{cut} $Y$). We denote by $\hh(Y)=\fp^{-1}(\ww(Y))$ the corresponding set of half-spaces. The sets ${\mathcal W}(Y)$ and $\hh(Y)$ are not \textit{a priori} measurable, unless $Y$ is countable for instance, in which case ${\mathcal W}(Y) =
\bigcup_{y,y'\in Y} \ww (y|y')$.} We write $\overline{\mu}({\mathcal W}(Y))\le K$ if for any measurable subset ${\mathcal E}\subseteq {\mathcal W}(Y)$ we have that $\mu({\mathcal W}({\mathcal E} ))\le K$, hence we may talk about a priori non-measurable sets having bounded measure.
\end{notation}
In this section we investigate under what circumstances a space with measured walls is within bounded Hausdorff distance of its associated median space, or a complete median space in general. Such a measured wall space would have to satisfy metric properties similar to those of a median space, up to bounded perturbation. We briefly recall here a few more properties of median spaces, as discussed in \cite{CDH-adv}, ending with the property that will be central to our approach.

\begin{defn}\label{def:proj}
Given a metric space $(X, \dist )$, a subset $A$ and a point $x$ in $X$, an {\em $\varepsilon$-projection $p$ of $x$ on $A$} is a point $p\in A$ such that $\pdist (x,p)< \pdist (x,A)+\varepsilon$.
A {\emph{gate}} between a point $x$ and a subset $Y$ of $X$ is a point $p\in Y$ that is between $x$ and any point $y\in Y$. A subset $Y$ is \emph{gate convex} if every point $x\in X\setminus Y$ has a gate.
\end{defn}
Every gate convex set is closed and convex. The converse is true if $X$ is a complete median space. 
\emph{In particular, in a complete median space the closure of every half-space is gate convex.} This is the property that we will quasify and use to discuss the finiteness of the Hausdorff distance of a measured wall space to its associated median space. 
\begin{defn}
A measured wall space $(X,\ww , \mu)$ is said to have \emph{quasi-gated half-spaces} if there are constants 
$\epsilon, K\geq 0$ such that every half-space $h$ of $X$ is $(\epsilon, K)$-\emph{gated}, namely, for every point $x$ outside $h$, and every $\epsilon$-projection $p$ of $x$ on $h$ (for the wall pseudo-distance), the set of walls separating $x$ from $p$ and intersecting $h$ has measure bounded by $K$.\end{defn}
\begin{rem}
The Euclidean plane $({\R}^2,\ell^2)$ does not have quasi-gated half spaces. Indeed, any straight line (or hyperplane) corresponds to a wall, and half-spaces are given by a side $h$ of that hyperplane. For a point $x\in h^c$, any $\epsilon$-projection $p$ has to be in the $\epsilon$-ball centered around the orthogonal projection of $x$ on the hyperplane $\bar{h}$. But the measure of the hyperplanes separating $x$ from $p$ and intersecting $\bar{h}$ is the same as the one of {\it all} hyperplanes separating $x$ from $p$, since the ones not intersecting $\bar{h}$ are the ones parallel to $\bar{h}$, which is a measure 0 set of hyperplanes. Hence that measure is almost equal to the distance between $x$ and $p$, so cannot be bounded by any constant independent of the distance between $x$ and $p$. This phenomena doesn't happen in the hyperbolic plane: there walls correspond to bi-infinite geodesics, and given a point $x$ and a wall $\bar{h}$, again any $\epsilon$-projection is in the $\epsilon$-ball around the nearest point projection of $x$ onto $\bar{h}$, but now if $x$ is far enough from $\bar{h}$, most geodesic passing near $x$ will fail to intersect $\bar{h}$, only the ones close enough to that $\epsilon$-projection will (see Figure \ref{fig:notgated}).
\begin{figure}
    \centering
    \includegraphics[width=0.5\textwidth]{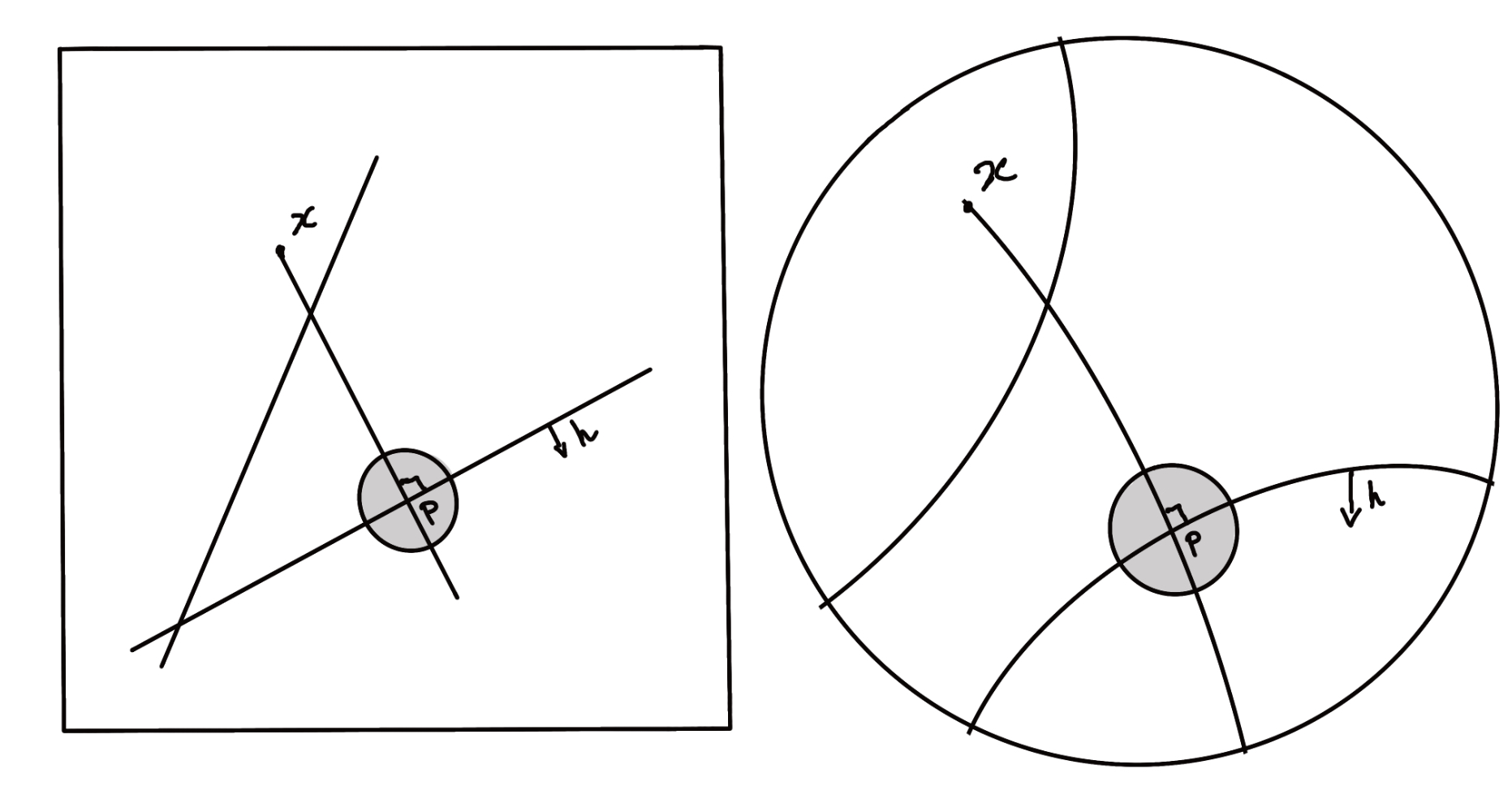}
    \caption{The Euclidean plane with most walls close to $x$ and crossing $h$, versus the hyperbolic plane where only hyperplanes near $p$ can cross $h$}
    \label{fig:notgated}
\end{figure}
\end{rem}
A consequence of the property that half-spaces are quasi-gated is a property of quasi-convexity of walls. The latter quasi-convexity property is slightly different from the one introduced in \cite{CDH-adv}, where we defined \emph{quasi-convex sets} to be sets $Y$ such that for every $a,b\in Y$, $I(a,b)\subseteq \onn_M (Y)$ for some uniform $M\geq 0$. Here, we extend the notion of quasi-convexity as follows.
\begin{defn}\label{quasi-conv}
Given a metric space $(X, \dist )$, and constants $\delta, M\geq 0$, a subset $Y\subseteq X$ is called $(\delta,M)$-\emph{quasi-convex} if, for every $a,b\in Y$, then
$$I_\delta (a,b)\subseteq \onn_M (Y).$$
\end{defn}
With this terminology, we obtain the following.
\begin{lem}\label{lem:smball3} Let $(X,{\mathcal W},\mu )$ be a measured wall space with quasi-gated half-spaces. Let $\epsilon, K\geq 0$ be constants such that all its half-spaces are $(\epsilon, K)$-gated.

Then, for every $\delta \geq 0$, all the half-spaces of $X$ are $(\delta , K+\delta )$-quasi-convex.
\end{lem}
\begin{proof}
Let $\delta \geq 0$, $x,y$ two distinct points in a half-space $h$ and $a\in I_\delta (x,y)$. Let $p$ be an $\epsilon$-projection of $a$ onto $h$. Then $\ww (a|p) \setminus \ww (a|h)$ has measure at most $K$.

On the other hand
$$\ww (a|h) \sqcup \ww (x|y) \subseteq \ww (a|x) \cup \ww (a|y),$$ whence 
$$\mu (\ww (a|h))+ \pdist_\mu (x,y)\leq \pdist_\mu (x,a) + \pdist_\mu (a,y) \Rightarrow \mu (\ww (a|h)) \leq \delta .$$
It follows that $\pdist_\mu (a,p)\leq K +\delta $.
\end{proof}
\begin{rem}\label{closedwallsQC}If we don't assume gated or quasi-gated, but we know instead that for any point $x\in X$ and any half-space $h\in\hh$ not containing the point $x$, the set $\ww(x|h)$ is measurable and its measure is equal to the pseudo-distance $\pdist_\mu (x,h)$ between $x$ and $h$, then the half-spaces are convex with respect to $\pdist_\mu$, in the sense that if $a$ and $b$ are two points in a half-space $h$ and $x\in I(a,b)$, then $\pdist_\mu (x,h) = \mu(\ww(x|h))=0$. Indeed, $x\in I(a,b)$ satisfies by definition that $\pdist_\mu (a,x) + \pdist_\mu (x,b) = \pdist_\mu (a,b)$. Hence, since in general $\pdist_\mu (a,b) = \mu (\ww (a,x\mid b))+ \mu (\ww (a\mid x, b))$, we have
$$2\mu(\ww(x|h))+\pdist_\mu (a,b)=\underbrace{\mu (\ww (a,x\mid b)) + \mu (\ww (x\mid h))}_{\mu(\ww (x\mid b))} +\underbrace{ \mu (\ww (a\mid x, b)) + \mu (\ww (x\mid h))}_{\mu(\ww (x\mid a))}$$
since $\ww (a,x\mid b)\sqcup\ww (x\mid h)\subseteq \ww (x\mid b)$, and similarly $\ww (a\mid x, b)\sqcup\ww (x\mid h)\subseteq \ww(x\mid a)$. Combining the above, we get $0=\pdist_\mu (a,x) + \pdist_\mu (x,b) - \pdist_\mu (a,b)\geq 2 \pdist_\mu (x,h)\geq 0$.  

The condition that $\pdist_\mu (x,h) = \mu(\ww(x|h))$ is satisfied for $\R^2$ when endowed with the $\ell^1$ metric, but not with the Euclidean metric, as in the latter case the hyperplanes are all the straight lines; we can have points arbitrarily far from a hyperplane, but the measure of the set of hyperplanes separating a point from another hyperplane is always 0, since this set can only contain hyperplanes of a fixed slope. \end{rem}
Having quasi-gated half-spaces also suffices to ensure that walls separating a point $x$ from a wall actually cut a small ball around $x$, which again is not true in the Euclidean plane.
\begin{lem}\label{lem:smball} Let $(X,\ww , \mu )$ be a measured wall space with the property that there are constants $\epsilon, K\geq 0$ such that all its half-spaces are $(\epsilon, K)$-\emph{gated}. 

Then for every $\tau\in\mm(X)$, and $x\in X$, $\epsilon$-projection of $\tau$ onto $X\subseteq \mm(X)$, each wall in $\ww (x\vert \tau)$ cuts the ball $\bar{B}(x,2K+\epsilon)\subseteq X$.
\end{lem}
\begin{proof}
Take $w=\{h,h^c\}\in\ww (x\vert \tau)$, so that $x\in h$ and $h^c\in \tau$. Let $p\in h^c$ be a point
such that $\pdist_\mu (p,x)\le \pdist_\mu (h^c,x)+\epsilon$. By assumption, we have that
$${\mu}({\mathcal W}(x\vert p)\setminus {\mathcal W}(x\vert h^c))\leq K\, .$$
We can write
$$
\pdist_\mu (\tau , \sigma_p)=\mu ((\tau \vartriangle \sigma_x)\vartriangle (\sigma_p \vartriangle
\sigma_x))= \mu (\ww (x|\tau ) \setminus \ww(x|p) )+ \mu (\ww (x|p)\setminus \ww (x|\tau ))\, .
$$
By the admissibility of $\tau$, $\ww (x|h^c)\subseteq \ww (x|\tau)$. It follows that $\ww (x|p)\setminus \ww (x|\tau ) \subseteq \ww (x|p)\setminus \ww (x| h^c )$, so that, using the assumption once again, we have $\mu(\ww (x|p)\setminus \ww (x|\tau ))\leq K$.

On the other hand 
\begin{eqnarray*}\mu \left(\ww (x|\tau ) \setminus \ww(x|p) \right) &=&\pdist_\mu (\tau , \sigma_x)-
\mu \left( \ww (x|\tau ) \cap \ww(x|p) \right)\\
&\leq &\pdist_\mu (\tau , \iota (X))+\epsilon- \mu \left( \ww
(x|h^c) \cap \ww(x|p) \right).\end{eqnarray*}
Now $\mu \left( \ww (x|h^c) \cap \ww(x|p) \right) = \mu (\ww (x|p))- \mu \left( \ww (x|p ) \setminus
\ww(x|h^c) \right)$, which by assumption is larger than $\pdist_\mu (x,p)- K$.
Combining all, we get that 
$$\pdist_\mu (\tau , \sigma_p)\leq K+\pdist_\mu (\tau , \iota (X))+\epsilon -\pdist_\mu (x,p)+K,$$
hence that $\pdist_\mu (x,p)\leq 2K + \epsilon$, as desired.
\end{proof}
This result is a quasification of a property that median spaces have, that {\it closed walls are gate convex}. It will suffice to conclude that $X$ is within bounded distance from $\mathcal{M}(X)$, provided that the set of walls intersecting a ball has finite measure. This is the condition that we introduce next.
\begin{defn}\label{defn:mulocfinite} A wall space $(X,{\mathcal W},\mu )$ is called \emph{$\mu$-locally finite} if for every $R>0$, there exists $M=M(R)\geq 0$ such that $\overline{\mu}({\mathcal
W}(B(x,R)))\le M<+\infty$ for every $x\in X$, where the open ball $B(x,R)$ is with respect to the wall metric $\pdist_\mu $.
 When we want to avoid specifying the measure $\mu$ and there is no risk of confusion, we simply say that the measured wall space is \emph{measurably locally finite}. 
\end{defn}
\begin{exmp}
\ 
\begin{enumerate}\item The measured wall space structures on $\R^n$ endowed with either the Euclidean norm or the $\ell^1$--norm are measurably locally finite, as is the structure on the real hyperbolic space ${\field H}^n$. 
\item Discrete wall spaces whose wall distance is uniformly proper are $\mu$-locally finite.
\item More generally, if for each ball we have a covering
$${\mathcal W}(B(x,R))\subseteq {\mathcal W}(x\vert a_1)\cup \dots\cup {\mathcal W}(x\vert a_{n(R)}),$$
with $d(x,a_i)\le g(R)$, then the measured wall space is measurably locally finite.\end{enumerate}
\end{exmp}
Lemma \ref{lem:smball} implies the following.
\begin{lem}\label{lem:smball2} Let $(X,{\mathcal W},\mu )$ be a measured wall space with quasi-gated half-spaces and $\mu$-locally finite. Then $\mm (X)$ is within Hausdorff distance from $\iota (X)$, and in particular is roughly median.

Precisely, if there are constants $\epsilon, K\geq 0$ such that all half-spaces are $(\epsilon, K)$-gated, then $\mm (X)$ is within Hausdorff distance at most 
$$\delta=\delta(K,\epsilon):=\bar{\mu} (\ww (\bar{B}(x, 2K+\epsilon))$$ 
from $\iota (X)$, and hence $X$ is $\delta$-median (Definition \ref{dmedian}).
\end{lem}
To complete the characterization of measured wall spaces within finite distance of their associated median space, we have to consider a converse of the statement in Lemma \ref{lem:smball2}. 
\begin{rem}
The property alone that $\mm (X)$ is within bounded distance from $\iota (X)$ does not suffice to imply that half-spaces are quasi-gated. Example \ref{trou} shows this, since the space $X$ is within bounded Hausdorff distance of $\mm (X)=M$, but for the midpoints $(a_n,0)$ of the geodesics $[x_n, y_n]\times \{0\}$ in $X$, there exist no points $p_n$ in the half-space $h=T\times [1/2,1]\cap X$ such that $\ww (a_n,p_n) \setminus \ww (a_n,h)$ has uniformly bounded measure. Indeed, for every $p_n\in h$, the measure of $\ww (a_n,p_n) \setminus \ww (a_n,h)$ is at least $n/2 - 1$. The reason for this is that in this example half-spaces are not quasi-convex in the sense of Definition \ref{quasi-conv}. Quasi-convexity is therefore a condition independent of the "bounded distance from $\mm (X)$" condition and, as it is also a necessary condition for the half-spaces to be quasi-gated, it has to be added to obtain a converse. This is the subject of the next lemma.    
\end{rem}
 It is to be noted that, when the quasi-convexity of the half-spaces is required, the condition "within bounded distance from $\mm (X)$" can be weakened to "tripodic".  
\begin{lem}\label{lem:convex} Let $(X,{\mathcal W},\mu )$ be a measured wall space that is $\delta$-tripodic when endowed with the wall pseudo-metric, for some $\delta\geq 0$, and such that there is $M\geq 0$ such that all the half-spaces are $(2\delta , M)$-quasi-convex.

For every half-space $h$ of $X$, every point $x$ outside $h$, and every $\epsilon$-projection $p$ of $x$ on $h$, the walls separating $x$ from $p$ and intersecting $h$ must intersect $\bar{B}(p, 2\delta+M+\epsilon)$. Namely, for every $k\geq 2\delta+M+\epsilon$ then
$$\ww(x|p)\cap\ww(h)\subseteq\ww(\bar{B}(p,k)).$$
In particular, if $X$ is $\mu$-locally finite, then all its half-spaces are $(\epsilon, K)$-gated for $K\geq  \bar{\mu }(\bar{B}(p, 2\delta+M+\epsilon ))$.
\end{lem}
\begin{proof}Let $x$ be a point outside a half-space $h_0$, and let $p$ be a $\epsilon$-projection of $x$ on $h_0$. Let ${\mathcal E}$  be a measurable subset of $\mathcal W$ contained in ${\mathcal W}(x\vert
p)\setminus {\mathcal W}(x\vert h_0)$. For any wall $w=\{h,h^c\}\in {\mathcal E}$ assume the notation is
such that $x\in h$ (so $p\in h^c$). Since the wall $w$ does not separate $x$ from $h_0$, the intersection
$h\cap h_0$ is not empty. Take $q\in h\cap h_0$. As $X$ is $\delta$-tripodic, there exists a point $m$ that is $2\delta$--between any two of the three points $x,p,q$. In particular, since the walls are $(2\delta , M)$-quasi-convex, $m\in I_{2\delta} (x,q)\subseteq \onn_M (h)$, and $m\in I_{2\delta}(p,q)\subseteq \onn_{M}(h_0)$ so 
$$\pdist(x,m)\geq \pdist(x,h_0)-M\geq\pdist(x,p)-\epsilon-M$$ 
hence $\pdist(x,p)-\pdist(x,m)\leq M+\epsilon$. Finally, from $m\in I_{2\delta}(x,p)$ we compute
$$\pdist(p,m)\leq\pdist(p,x)-\pdist(m,x)+2\delta\leq 2\delta+M+\epsilon.$$
This shows that $w\in {\mathcal W}(\bar{B}(p,\epsilon +M +2\delta ))$. 
\end{proof}
Combining all the preceding lemmas, we obtain the following result, that immediately implies Theorem \ref{thm:main}.
\begin{thm}\label{thm:equiv}
Let $(X,\mathcal{W},\mu)$ be a space with measured walls that is $\mu$-locally finite (see Definition \ref{defn:mulocfinite}). The following are equivalent:
\begin{enumerate}
\item\label{Haus} there are constants $\delta, M\geq 0$ such that the Hausdorff distance from $\iota(X)$ to the associated median space $\mm (X)$ is at most $\delta$ and all the half-spaces of $X$ are $(2\delta , M)$-quasi-convex;
\item\label{HausGen} there are constants $\delta, M\geq 0$ such that there exists a median space ${\mathcal{M}}$ and a monomorphism $\varphi :X \to {\mathcal{M}}$ such that ${\mathcal{M}}$ is within Hausdorff distance at most $\delta$ from $\varphi (X)$, and all the half-spaces of $X$ are $(2\delta , M)$-quasi-convex;
\item\label{(3)} there exists $\delta$ and $M$ non-negative constants such that $X$ with the wall pseudo-metric is $\delta$-tripodic, and all the half-spaces of $X$ are $(2\delta , M)$-quasi-convex; 
%
\item\label{convex} the space $X$ has $(\epsilon ,K)$-gated half-spaces, for some $\epsilon$ and $K$ non-negative constants.
\end{enumerate}
\end{thm}
\begin{proof}
    That (\ref{Haus}) implies (\ref{HausGen}) is immediate, and that (\ref{HausGen}) implies (\ref{(3)}) follows from the fact that roughly median spaces are in particular tripodic. That (\ref{(3)}) implies (\ref{convex}) is the content of Lemma \ref{lem:convex}, and that (\ref{convex}) implies (\ref{Haus}) follows from the combination of Lemma \ref{lem:smball2} and Lemma \ref{lem:smball3}. 
\end{proof}

Note that the property of being at finite Hausdorff distance from its median space is not inherited by subsets, not even when they are geodesic. Indeed, every $L^2$ space can be embedded isometrically into an $L^1$ space \cite{WellsWilliams:Embeddings}. 

A particular case of the equivalence in Theorem \ref{thm:equiv} is the following.
\begin{thm}\label{thm:equiv2}
Let $(X,\mathcal{W},\mu)$ be a space with measured walls that is $\mu$-locally finite and has quasi-convex half-spaces. The following are equivalent:
\begin{enumerate}
\item\label{Haus2} the Hausdorff distance from $\iota(X)$ to the associated median space $\mm (X)$ is finite;
\item\label{HausGen2} there exists a median space ${\mathcal{M}}$ and a monomorphism $\varphi :X \to {\mathcal{M}}$ such that ${\mathcal{M}}$ is within finite Hausdorff distance from $\varphi (X)$;
\item\label{(3)-2} there exists $\delta$ such that the pseudo-metric space $(X, \pdist_\mu)$ is $\delta$-tripodic; 
%
\item\label{convex2} there are $\epsilon$ and $K$ non-negative constants such that the space $X$ has $(\epsilon ,K)$-gated half-spaces.
\end{enumerate}
\end{thm}
\begin{cor}\label{cor:yd}
A space with measured walls that is $\mu$--locally finite and \eff $\delta$-tripodic (e.g. geodesic Gromov hyperbolic or finite product of geodesic Gromov hyperbolic spaces) with respect to the distance $\pdist_\mu$ is at finite Hausdorff distance of its associated median space. 
\end{cor}
 \begin{cor}\label{cor:hypreal}
A finite product of finite dimensional real hyperbolic spaces is within bounded distance from its associated median space.
\end{cor}
\begin{rem}\label{infdim} The infinite dimensional real hyperbolic space can be described in several ways (see Example \ref{exmp:realhyper} for the finite dimensional case). For instance, given the Hilbert space $\ell^2$ and $O(1,\infty )$ the space of bounded operators preserving the form $-x_1^2 +\sum_{i=2}^\infty x_i^2$, one can define $\field H^\infty$ as the infinite dimensional Riemannian symmetric space of constant negative curvature $O(1,\infty )/ [O(1) \times O(\infty )]$, where $O(\infty )$ represents the space of orthogonal operators that keep the first coordinate $x_1$ fixed \cite[$\S 6.A.III$]{Gro}. Alternatively, given a Hilbert space $H$ and a $1$-dimensional subspace $L$ in it, one can define a bilinear symmetric form determined by the quadratic form  
 $$
 B(v,v)= \|v_L\|^2 - \|v_{L^\perp}\|^2, 
 $$ where $v_L$ and $v_{L^\perp}$ are the orthogonal projections of $v$ on $L$, respectively on its orthogonal $L^\perp$. For a fixed non-zero vector $e$ in $L$ one can define the hyperboloid model of the infinite dimensional hyperbolic space as
 $$
 {\field H}^\infty = \{ v\in H \mid B(v,v)=1, B(v,e) >0 \},
 $$ with distance defined by $\cosh \left(\dist_{{\field H}^\infty} (x,y) \right) = B(x,y)$, turning it into a complete metric space with constant curvature $-1$ (in an appropriate sense) \cite{MonodPy14}. Note that there is a generalization of this construction, $\field H^\alpha$, of Hilbert dimension $\alpha$, for every cardinal $\alpha \geq 2$ \cite{BIM}. For simplicity, here we restrict to $\alpha = \aleph_0$.
 
 Finally, ${\field H}^\infty$ can also be defined as the direct limit of the sequence of metric spaces $\left({\field H}^n \right)_{n\in \N }$, where ${\field H}^n\hookrightarrow {\field H}^{n+1}$ is the natural inclusion obtained by adding a 0 direction.
 

It is unclear how to endow the space $X= \field H ^\infty$ with a measured walls structure, and even if it were, only a few of the previous arguments would work. In particular, the following are true:
\begin{enumerate}
    \item for every $\lambda >0$, the coarse interval $I_\lambda (x,y)$ is contained in $\nn_{\lambda +\delta } (I(x,y))$, for every $x,y$, where $\delta $ only depends on the hyperbolicity constant. Therefore, all the half-spaces are quasi-convex in the most general sense.
    \item the half-spaces are quasi-gated only in the following sense: given a half-space $h$, a point $x$ outside $h$ and an $\epsilon$-projection $p$ of $x$ onto $h$, every wall in $\ww(x|p) \setminus \ww (x|h)$ intersects a ball $B(p, \delta )$, with $\delta $ depending only on the hyperbolicity constant and on $\epsilon$ (Lemma \ref{lem:convex}). 
    \end{enumerate}
However, as $X$ is not measurably locally finite, the rest of the argument fails. It is nevertheless natural to ask whether the result cannot be obtained by different means. An answer to Question \ref{ques:infdim} cannot be obtained by taking a limit of the finite dimensional case either, as the bound we obtain for the distance to the associated median space in that case depends on volumes of balls, that increase with the dimension.
\end{rem}%
\begin{rem}\label{rem:mwallm(X)}
To complete the picture, we note that for a measured wall space $X$ satisfying the equivalent assumptions of Theorem \ref{thm:equiv}, a connection can be established between its half-spaces $h$ and corresponding sets $h_\mm$ of $\mm (X)$, that are not half-spaces in general, defined, for each $h\in\hh$, by
$$h_\mm=\{[\sigma]\in\mm(X)\hbox{ such that }h\in\sigma\hbox{ for some }\sigma\in[\sigma]\}.$$
These sets are convex, but their complements in $\mm(X)$ might not be. Note that the complement of a set $h_\mm$ is in general not $(h^c)_\mm$, since we can have ${h_\mm}\cap(h^c)_\mm\not=\emptyset$, see Remark 2.18 in \cite{F2}.
\end{rem}
%

%
%

\section[Local compactness of medianization of real hyperbolic spaces]{Local compactness of the medianization of real hyperbolic spaces.}\label{RhypLocCpct}
In this section, we will show that the medianization of a real hyperbolic space is locally compact. To do so, we will show that balls in that medianization are {\it totally bounded}, namely they can be covered by finitely many balls of any given radius. In a metric space, a subset is compact if and only if it is complete and totally bounded. More precisely, the aim of this section is to show the following.
\begin{prop}\label{pro:totallybded} Let $n\geq 1$. For any $R\geq 0$ and any $\varepsilon>0$, any ball of radius $R$ in $\mm({\mathbb H}^n)$ can be covered by a finite collection of balls of radius $\varepsilon$.\end{prop}
We start by identifying the walls in ${\mathbb H}^n$.
\begin{notation}\label{def:wallsHn}
Fix a basepoint $x_0$ in ${\mathbb H}^n$. For every point $x\neq x_0$, we denote by $H_x$ the unique hyperplane orthogonal to the unique geodesic $[x, x_0]$ and containing $x$, and by $W_x$ the unique wall determined by $H_x$ and whose closed component contains $x_0$. 
\end{notation}
\begin{lem}\label{lem:theta}Let $n\geq 2$. For every admissible section $\tau\in\mm(\mathbb{H}^n)$ and non-trivial geodesic $[a,x_0]$ in ${\mathbb H}^n$, one of the following cases occurs:
\begin{enumerate}
\item\label{c1} the exists $\theta \in (a,x_0)$ such that $\tau $ coincides with $\sigma_a$ on every $W_x$ with $x\in (\theta , x_0]$ and $\tau $ coincides with $\sigma_{x_0}$ on every $W_x$ with $x\in [a, \theta )$; 
\item\label{c2} $\tau $ coincides with $\sigma_a$ on every $W_x$ with $x\in (a ,x_0 )$;
\item\label{c3} $\tau $ coincides with $\sigma_{x_0}$ on every $W_x$ with $x\in (a ,x_0 )$.
\end{enumerate}
\end{lem} 
\proof Suppose that we are neither in case \eqref{c2} nor in case \eqref{c3}. Thus, there exists a point $u\in (a ,x_0 )$ such that $\tau $ coincides with $\sigma_a$ on $W_u$ and a point $v\in (a ,x_0 )$ such that the section $\tau$ coincides with $\sigma_{x_0}$ on $W_v$. Clearly for every $x\in [u, x_0]$, the section $\tau$ coincides with $\sigma_a$ on $W_x$, and for every $y\in [a, v]$, it coincides with $\sigma_{x_0}$ on $W_y$. In particular $u$ must be between $v$ and $x_0$. Let $\delta$ be the supremum of the distances to $x_0$ of points such as $u$ and let $\theta$ be the point on $[a, x_0]$ that is at distance $\delta $ from $x_0$. As $\theta $ appears as limit of a sequence of points $u_n$ as above and $(\theta , x_0] = \bigcup [u_n , x_0 ]$, the first property in  \eqref{c1} follows, and the second property follows according to the choice  of $\theta $ and the fact that for every $x\in (a,x_0)$, then $\tau $ must equal either  $\sigma_{a}$ or  $\sigma_{x_0}$.    \endproof 
\begin{notation}\label{not:switchmap}
We denote by $\theta_\tau (a)$ the point $\theta $ in case \eqref{c1}, the endpoint $a$ in case \eqref{c2} and the endpoint $x_0$ in case \eqref{c3}. Note that when $a$ varies on the sphere $S(x_0, R)$, the parameter $\theta_\tau (a)$ defines a map 
\begin{eqnarray*}\delta_\tau : S(x_0, R) &\to& [0,R]\\
a & \mapsto& \dist (x_0, \theta_\tau (a)).\end{eqnarray*}
\end{notation}
\begin{lem}\label{lem:cont}
For every admissible section $\tau\in\mm(\mathbb{H}^n)$, the map $\delta_\tau$ defined above is continuous. 
\end{lem}
\proof Let $a_0$ be a fixed point on $S(x_0, R)$ and let $a$ be a point on $S(x_0, R)$ such that the angle between $[x_0, a]$ and $[x_0, a_0]$ is at most $\epsilon$. For every point $y_0$ on $[a_0, x_0]$, the point $y$ on $[a, x_0]$ that is nearest to $a$ such that $H_y$ does not intersect $H_{y_0}$ satisfies 
$$
\dist(x_0, y) \leq \dist(x_0, y_0) + \kappa \epsilon R,
$$ for a universal constant $\kappa $. 
This implies that $\dist(x_0, \theta_\tau (a))\leq \dist(x_0, \theta_\tau (a_0))+ \kappa \epsilon R.$ The opposite inequality is obtained by swapping $a_0$ and $a$ in the previous argument.
\endproof
\begin{rem}\label{rem:ProjClose}
    According to Lemmas \ref{lem:convex} and \ref{lem:smball}, given $\tau$ an arbitrary admissible section in $\mm(\mathbb{H}^n)$, and $x\in \mathbb{H}^n$ a point such that
$\pdist_{\mu}(\tau ,\sigma_x)\le \pdist_{\mu} (\tau ,\iota (\mathbb{H}^n))+1$, any wall that separates $x$ from $\tau$ also cuts the ball $B(x,\rho)$, where $\rho$ only depends on the hyperbolicity constant and the dimension of $\mathbb{H}^n$. This implies that (in the Notation \ref{Nwallssubsets} where $\hh(Y)$ is the set of half-spaces cutting $Y$)
$$
\tau = \left[ \sigma_x \cap \hh (B(x, \rho ))^c\right] \sqcup \left[\tau \cap \hh (B(x, \rho )) \right],
$$
meaning that $\tau$ and $\sigma_x$, as admissible sections, coincide on the walls that do not intersect $B(x, \rho )$ but may differ on the walls intersecting that ball, see Figure \ref{fig:ProjClose}. In particular, 
$$\dist_{\mm(X)} (\tau , \sigma_x )\leq\mu(\hh (B(x, \rho ))):=\Delta_\rho.$$
\begin{figure}
    \centering
    \includegraphics[width=0.5\textwidth]{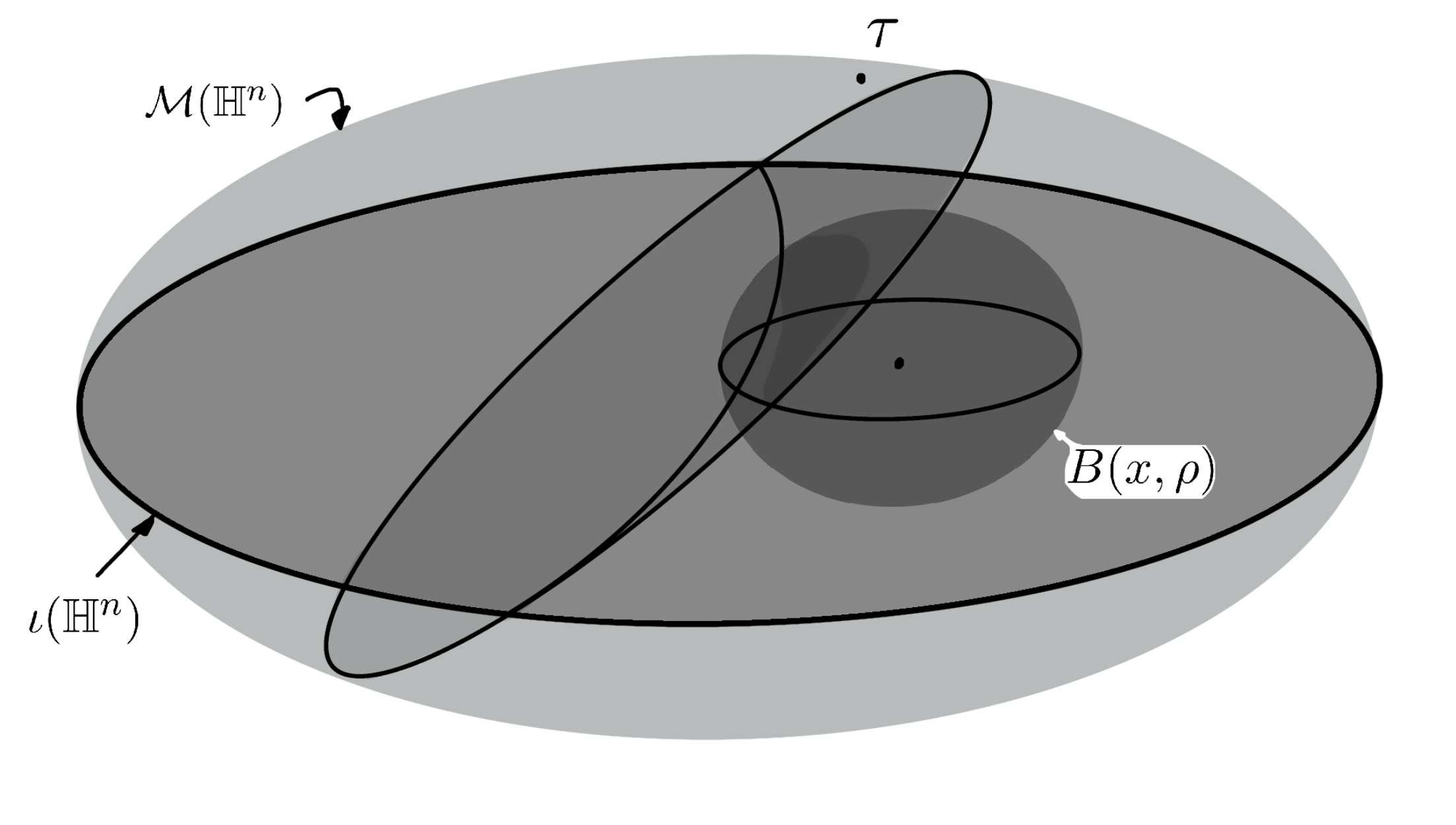}
    \caption{A wall separating $\tau$ from a 1-projection $x$ in $\iota (\mathbb{H}^n)$ also cuts the ball $B(x,\rho)$}
    \label{fig:ProjClose}
\end{figure}
\end{rem}
We now have everything in place for the proof of Proposition \ref{pro:totallybded}.
\begin{proof}[Proof of Proposition \ref{pro:totallybded}.]
According to Corollary \ref{cor:hypreal}, the median space $\mm({\mathbb H}^n)$ is at finite Hausdorff distance from ${\mathbb H}^n$, so it suffices to consider balls centered at some $\sigma_{x_0}$, for $x_0\in{\mathbb H}^n$ fixed. Let $\tau$ be an arbitrary admissible section in this ball $B(\sigma_{x_0},R)$, any 1-projection in $\iota({\mathbb H}^n)\subseteq \mm({\mathbb H}^n)$ of that given $\tau$ is within distance $R+\Delta_\rho$ from $x_0$. In particular, for $R'= R+ \Delta_\rho + \rho$ we can write (see Remark \ref{rem:ProjClose} above)
$$
\tau = \left[ \sigma_{x_0} \cap \hh (B(x_0, R' ))^c\right] \sqcup \left[\tau \cap \hh (B(x_0 , R' )) \right]. 
$$
meaning that $\tau$ coincides with $\sigma_{x_0}$ on walls that are far enough from $x_0$. 

To show that the ball of radius $R$ is totally bounded, it suffices to find, for every $\epsilon >0$, a finite number of admissible sections $\tau_1,\dots , \tau_m$ in the closed ball of radius $R$ around $\sigma_{x_0}$ such that for every admissible section $\tau$ at distance at most $R$ from $\sigma_{x_0}$, there exists some $\tau_i$ such that the symmetric difference 
$$
\tau \vartriangle \tau_i = \left[\tau \cap \hh (B(x_0 , R' )) \right] \vartriangle \left[\tau_i \cap \hh (B(x_0 , R' )) \right] (=\pdist(\tau,\tau_i))
$$ 
has measure at most $C\epsilon$, where $C$ is a constant depending on the hyperbolicity constant and the dimension of ${\mathbb H}^n$. To construct those admissible sections, we consider $a_1, \dots , a_k$ points on the sphere $S(x_0, R')$ such that for every $a\in S(x_0, R')$ there exists an $a_i$ such that the angle between $[x_0, a]$ and $[x_0, a_i]$ is at most $\epsilon/2R'$. For each $i\in \{ 1,2, \dots , k \}$ we consider $b_1(i), \dots, b_n(i)$ on $[x_0, a_i]$ ordered from $x_0$ to $a_i$, dividing it into subgeodesics of equal length $\epsilon$. This process subdivides the ball of radius $R'$ into a grid of mesh $\epsilon$.

On the finite set of walls 
$$W_F=\bigcup_{i=1}^k\bigcup_{j=1}^n\{W_{b_j(i)}\}$$ 
the number of admissible sections is finite. Let  $\tau_1,\dots, \tau_m$ be  admissible sections at distance at most $R$ from $\sigma_{x_0}$ such that all the possible restrictions $\tau|_{W_F}$ of an admissible section at distance at most $R$ from $\sigma_{x_0}$ are realized by some $\tau_\ell$. In particular, for any $\tau$ at distance at most $R$ from $\sigma_{x_0}$ there exists $\tau_\ell$ such that $\tau$ restricted to $W_F$ coincides with $\tau_\ell$. 

We are left to show that the distance from $\tau $ to $\tau_\ell$ in $\mm (X)$ is at most $C \epsilon$. To this end, assume that $\theta_{\tau } (a_r)$ is inbetween $b_j (i)$ and $b_{j+1} (i)$, possibly equal to $b_j (i)$. Then, on the set $W_{b_{j+1} (i)}$, the section  $\tau$ coincides with $\sigma_{x_0}$ and on the set $W_{b_{j}(i)}$, the section $\tau$ coincides with $\sigma_{a}$. Whence the same is true for $\tau_\ell$ and $\theta_{\tau_\ell} (a_r)$ is inbetween $b_j(i)$ and $b_{j+1} (i)$. It follows that the distance between  $\theta_{\tau} (a_r)$ and  $\theta_{\tau_\ell}(a_r)$ is at most $\epsilon$. 

For an arbitrary $a\in S(x_0, R')$, the above and Lemma \ref{lem:cont} imply that the distance between  $\theta_{\tau } (a)$ and  $\theta_{\tau_\ell} (a)$ is at most $(1+\kappa)\epsilon$, for the constant $\kappa$ of Lemma \ref{lem:cont} above.

Thus, for every $a\in  S(x_0, R')$, the sections $\tau$ and $\tau_\ell$ coincide on every $W_x$ with $x\in [a, x_0]$, except maybe on a small sub-interval of length at most $(1+\kappa)\epsilon$. The distance between $\tau$ and $\tau_\ell$, given by the measure of the walls on which they differ, is hence bounded by $C\epsilon$, where $C$ depends on the dimension $n$. 
\end{proof}
Proposition \ref{pro:totallybded} combined with Corollary \ref{cor:hypreal} imply the following. 
 \begin{cor}\label{cor:latticeshypreal}
Let $X$ be a  finite product of finite dimensional real hyperbolic spaces.
Every uniform lattice in ${\mathrm{Isom }} (X)$ acts on the metric completion $\overline{\mm (X)}$ of the median space associated to $X$ properly discontinuously and with compact quotient.
\end{cor}
\begin{rem}
It would be interesting to know if $\mm(X)$ is already complete, but this is not immediate in our opinion. It follows from results of \cite{F1,F4} that irreducible lattices as in Corollary \ref{cor:latticeshypreal} cannot act properly discontinuously cocompactly on a median space of finite rank. This implies that:
\begin{itemize}
\item the median space $\mm (X)$ has infinite rank;
\item the action described in Corollary \ref{cor:latticeshypreal} is the best type of action on a median space that can be found for such lattices. 
\end{itemize}
\end{rem}
%
\section{The complex hyperbolic space.}\label{Chyp}
In the case of complex hyperbolic spaces $\field H^n_\C$, their hyperbolic metric $\dist_{{\mathbb{H}}_{\mathbb{C}}}$ cannot be induced by a wall structure (see the proof of Corollary \ref{HC}, at the end of this section). The square root of this metric however is known to come from a wall structure (possibly also larger powers $\dist_{{\mathbb{H}}_{\mathbb{C}}}^\alpha$ with $\alpha \in (1/2 ,1)$ might, but nothing is proven in this respect). Here we explain that no metric $\dist_{{\mathbb{H}}_{\mathbb{C}}}^\alpha$ can be $\delta$-tripodic, hence the wall space it would be induced by cannot be within bounded Hausdorff distance from a median space. This is explained in Proposition \ref{prop:snowflake} below, which follows from simple considerations about snowflaked metrics. For the sake of completeness, we provide the entire (easy computational) argument here.
\begin{lem}\label{lem:alphabeta}
Let $a\geq b \geq 0$, let $0< \alpha <1$ and $\beta >1$. The following inequalities hold: 
\begin{enumerate}
\item\label{beta} $(a+b)^\beta \geq a^\beta + b^\beta$;
\item\label{alpha}  $ a^\alpha + b^\alpha - (a+ b)^\alpha \geq b^\alpha (2-2^\alpha )\geq 0$. 
\end{enumerate}
\end{lem}
\proof The inequalities are trivial when $b=0$. Thus we may assume that $b>0$, in fact  without loss of generality we may assume that $b=1$ and $a\geq 1$. 

\eqref{beta} follows from the fact that the function $f(x)= (x+1)^\beta -x^\beta -1$ is increasing for $x\geq 1$ and $f(1)= 2^\beta -2 >0$.

\eqref{alpha} follows from the fact that the function $g(x)= x^\alpha +1 -(x+1)^\alpha$ is increasing for $x\geq 1$ and $g(1)= 2- 2^\alpha$.
    
\endproof
\begin{prop}[Snowflaked metric spaces]\label{prop:snowflake}
Let $(X, \dist)$ be a metric space and let $0<\alpha < 1$. This implies that $\dist^\alpha$ is a metric, and the following holds. 
\begin{enumerate}
\item\label{interval} For every two points $x,y$ in $X$, the interval with respect to the metric $\dist^\alpha$ reduces to $\{ x,y\}$.
\item\label{snow} If $(X, \dist^\alpha )$ is $\delta$--tripodic, for some constant $\delta >0$, then $X$ has bounded diameter. 
\end{enumerate}
\end{prop}
\proof That $\dist^\alpha$ is a metric follows from Lemma \ref{lem:alphabeta}.

\eqref{interval} We denote the distance $\mbd=\dist^\alpha$, and set $\beta=\frac{1}{\alpha }$. We thus have that $\dist = \mbd^\beta$.

Let $z$ be a point between two points $x,y$ in $X$, with respect to the metric $\mbd$. Thus $\mbd (x,y) = \mbd (x,y) + \mbd (y,z)$, whence
\begin{equation}\label{eq:snow}
\dist (x,y) = \left[ \mbd (x,z) + \mbd (z,y) \right]^\beta \geq \dist (x,z) + \dist (z,y) \geq \dist (x,y).
\end{equation}
The first inequality in \eqref{eq:snow} follows from Lemma \ref{lem:alphabeta}\eqref{beta}. Note that equality holds if and only if one of the two numbers is zero. Since the first and the last terms in \eqref{eq:snow} are equal, all inequalities become equalities, in particular from the above we can deduce that either $\mbd (x,z) =0$ or $\mbd (z,y) =0$.   

\medskip

\eqref{snow} The proof is a coarse version of the proof of \eqref{interval}. Take three points $x,y,z\in X$ and  consider a $\delta$--median point $m$ with respect to the metric $\mbd$. Without loss of generality we can assume that $\dist (m,x)\leq \dist (m,y)\leq \dist (m,z)$. We can write
\begin{equation}\label{eq:alpha}
\mbd (x,y) + \delta \geq \mbd (x,m)+ \mbd (m,y)\geq \left(\dist (x,m) + \dist (m,y) \right)^\alpha \geq \dist (x,y)^\alpha = \mbd (x,y). 
\end{equation}
The second inequality is Lemma \ref{lem:alphabeta} \eqref{alpha}. Using  \eqref{eq:alpha} and again Lemma \ref{lem:alphabeta}\eqref{alpha}, we can write that
\begin{equation}\label{eq:alpha2}
\delta \geq \dist (x,m)^\alpha + \dist (m,y)^\alpha - \left[\dist (x,m) + \dist (m,y) \right]^\alpha \geq (2-2^\alpha ) \dist (m,x). 
\end{equation} 
If we repeat the arguments in \eqref{eq:alpha} and \eqref{eq:alpha2}, with $x,y$ replaced by $y,z$, we obtain that
$$
(2-2^\alpha ) \dist (m,y) \leq \delta ,
$$ whence 
$$
\dist (x,y)\leq \frac{2\delta }{2-2^\alpha }.
$$
This shows that only close enough triples of points admit a $\delta$--median point, and hence $(X,\mbd)$ is $\delta$-tripodic only if it has bounded diameter, that is only when $(X,\dist)$ has bounded diameter.
\endproof

\begin{cor}\label{cor:snow}
Let $(X, \dist )$ be an unbounded geodesic Gromov hyperbolic metric space, and let $0<\alpha < 1$. Then $(X, \dist^\alpha )$ cannot be $\delta$--tripodic for any $\delta >0$.

In particular if the snowflaked metric $\dist^\alpha $ is induced by a measured walls structure, then no isometric embedding of $(X, \dist^\alpha )$ into a median metric space $\mm$ has image within bounded distance of $\mm$.\end{cor}
Consider now the hyperbolic space $\field H^n_\C$ with $n\geq 2$, endowed with the hyperbolic distance $\dist$. Recall that $\field H^n_\C$ admits a structure of space with measured walls such that the induced distance is $\dist^{1/2}$. Possibly the exponent $1/2$ can be increased to some other value $\alpha \in (1/2,1)$. In either case, the space with measured walls thus obtained cannot be within finite distance of a median space due to Proposition \ref{prop:snowflake}. Thus we can finish the proof of Corollary \ref{HC} stated in the introduction.  
\begin{proof}[Proof of Corollary \ref{HC}] \eqref{nowalls}\quad An isometric embedding of any metric space $Y$ in a space with walls $X$ gives a wall structure on $Y$ (by restricting the walls of $X$ to the isometrically embedded subspace $Y$). But a wall on a manifold gives a codimension one totally geodesic submanifold as the intersection of the topological closure of each side. However, according to \cite[$\S  3.1.11$]{Gol}, every totally geodesic submanifold of $\field H^n_\C$ is either totally real or complex-linear, and in particular has real-codimension at least 2 when $n\geq 2$. See also Remark 2.11 of \cite{PSZ}.

\medskip

\eqref{complexsnow} is an immediate consequence of Corollary \ref{cor:snow}.\end{proof}


\begin{thebibliography}{dCTV06}

\bibitem[ArP]{ArntPhD}S. Arnt, \emph{Large scale geometry and isometric actions on Banach spaces}, Ph.D thesis, Universit\'e d'Orl\'eans, July 2014, https://theses.hal.science/MSL-THESE/tel-01088009v1
\bibitem[BH83]{BandeltHedlikova}
H.-J.~Bandelt and J.~Hedlikova, \emph{Median algebras}, Discrete Math.
  \textbf{45} (1983), 1--30.
  
\bibitem[Ber]{Be}N. Bergeron, \emph{private communications}.

\bibitem[BC]{BC} N. Bergeron and L. Clozel, \emph{Quelques cons\'equences des travaux d'Arthur pour le spectre et la topologie des vari\'et\'es hyperboliques},  Invent. Math. \textbf{192} (2013), 505--532. 
  
\bibitem[BeWi]{BW} N. Bergeron and D. Wise, \emph{A boundary criterion for cubulation}, American Journal of Mathematics \textbf{134} (2012), 843--859.

\bibitem[BBF15]{BBF-IHES} M.~Bestvina, K.~Bromberg and K.~Fujiwara, \emph{Constructing group actions on quasi-trees and applications to
   mapping class groups}, Publ. Math. Inst. Hautes \'{E}tudes Sci., {\bf{122}} (2015), 1--64.

\bibitem[BBF21]{BBF} M.~Bestvina, K.~Bromberg and K.~Fujiwara, \emph{Proper actions on finite products of quasi-trees},
Annales Henri Lebesgue, Volume {\bf 4} (2021), 685--709.

\bibitem[BBFS]{BBFS} M.~Bestvina, K.~Bromberg, K.~Fujiwara and A.~Sisto, \emph{Acylindrical actions on projection complexes}, Enseign. Math., Volume {\bf 65} (2019) no. 1-2, 1--32.

\bibitem[BF95]{BestvinaFeighn:stableactions}
M.~Bestvina and M.~Feighn, \emph{Stable actions of groups on real trees},
  Invent. Math. \textbf{121} (1995), no.~2, 287--321.
  
\bibitem[Bes]{Best} M.~Bestvina, \emph{$\R$--trees in topology, geometry, and
group theory}, Handbook of geometric topology, 55--91,
North-Holland, Amsterdam, 2002.

\bibitem[Bo13]{Bow13} B.~H.~Bowditch, \emph{Coarse median spaces and groups}, Pac. J. Math. \textbf{261} (2013), 53--93.
  
\bibitem[Bo14]{Bow14} B.~H.~Bowditch, \emph{Embedding median algebras in products of trees},
Geom. Dedicata \textbf{170} (2014), 157--176. 

\bibitem[Bo16]{Bow16} B.~H.~Bowditch, \emph{Some properties of median metric spaces},
Groups, Geom. Dyn. \textbf{10} (2016), 279--317.

\bibitem[Bo20]{Bow20} B.~H.~Bowditch, \emph{Median and injective metric spaces}, Math. Proc. Cambridge Philos. Soc. \textbf{168} (2020), 43--55.

\bibitem[BH]{BridsonHaefliger} M.~R.~Bridson and A.~Haefliger, \emph{Metric spaces of non-positive curvature}, series \textit{Grundlehren der mathematischen Wissenschaften [Fundamental Principles of Mathematical Sciences]}, vol. \textbf{319}, Springer-Verlag, Berlin, 1999, pp. xxii+643.

\bibitem[BIM05]{BIM} M.~Burger, A.~Iozzi and N.~Monod, \emph{Equivariant embeddings of trees into hyperbolic spaces}, Int. Math. Res. Not. \textbf{22} (2005), 1331--1369.

\bibitem[Ca11]{CantatAnn11}
S.~Cantat, \emph{Sur les groupes de transformations birationnelles des surfaces}, Ann. of Math. (2), \textbf{174} (2011), 299--340.

\bibitem[CM15]{CapraceMonod} P.-E.~Caprace and N.~Monod, \emph{An indiscrete Bieberbach theorem: from amenable CAT(0) groups to Tits buildings} J. Ec. polytech. Math. \textbf{2} (2015), 333--383. 


\bibitem[CCHO]{CCHO}J.~Chalopin, V.~Chepoi, H.~Hiroshi, D.~Osajda, \emph{Weakly modular graphs and non-positive curvature}, Memoirs of the American Mathematical Society, American Mathematical Society \textbf{268} (1309), 2020, pp.vi+159.
  
\bibitem[CDH10]{CDH-adv}
I.~Chatterji, C.~Dru\c{t}u, and F.~Haglund, \emph{{Kazhdan and Haagerup properties from the median viewpoint}}, Adv. in Mathematics \textbf{225} (2010), 882--921.
 
 \bibitem[CFI]{CFI}I. Chatterji, T. Fern\'os, and A. Iozzi, with an appendix by P.-E. Caprace,
 \emph{The Median Class and Superrigidity of Actions on CAT(0) Cube Complexes}, Journal of Topology \textbf{9} (2016), 349--400. 

\bibitem[Che00]{Chepoi:graphs}
V.~Chepoi, \emph{{Graphs of Some CAT(0) Complexes}}, Adv. in Appl. Math.
  \textbf{24} (2000), 125--179.
  
\bibitem[BC08]{ChepoiBandelt}
H.-J.~Bandelt and V.~Chepoi, \emph{Metric graph theory and
  geometry: a survey}, Surveys on discrete and computational geometry, Contemp.
  Math., vol. 453, Amer. Math. Soc., Providence, RI, 2008, 49--86.


\bibitem[CMV04]{CherixMartinValette}
P.A. Cherix, F.~Martin, and A.~Valette, \emph{Spaces with measured walls, the
  {H}aagerup property and property {(T)}}, Ergod. Th. \& Dynam. Sys.
  \textbf{24} (2004), 1895--1908.

\bibitem[dCTV06]{CornTessValette:lp}
Y.~de~Cornulier, R.~Tessera, and A.~Valette, \emph{{Isometric group actions on
  Banach spaces and representations vanishing at infinity}}, Transform. Groups \textbf{13} (2008), no. 1, 125--147.

\bibitem[DS05b]{DrutuSapir:TreeGraded}
C.~Dru\c{t}u and M.~Sapir, \emph{Tree-graded spaces and asymptotic cones of groups}, Topology
  \textbf{44} (2005), 959--1058, {with an appendix by D. Osin and M. Sapir}.

\bibitem[Dr09]{Drutu} 
C.~Dru\c{t}u, \emph{Relatively hyperbolic groups: geometry and quasi-isometric invariance}, Comment. Math. Helv. \textbf{84} (2009), 503--546. 


\bibitem[dV93]{VandeVel:book}
M.~Van~de~Vel, \emph{{Theory of Convex Structures}}, Elsevier Science Publishers,
  Amsterdam, 1993.

\bibitem[Eld]{E} M. Elder, \emph{$L_\delta$ groups are almost convex and have sub-cubic Dehn function},
Algebraic and Geometric Topology Vol 4 (2004) 23--29.

\bibitem[FH74]{FarautHarz}
J.~Faraut and K.~Harzallah, \emph{Distances hilbertiennes invariantes sur un espace homog\`ene}, Ann. Institut Fourier \textbf{3} (1974), no.~24,
  171--217.
\bibitem[Fio$_1$]{F1} E. Fioravanti, \emph{Tits Alternative and superrigidity for groups acting on median spaces}, Transfer of status thesis, University of Oxford, March 2017.

\bibitem[Fio$_2$]{F2} E. Fioravanti, \emph{Roller boundaries for median spaces and algebras}, Algebr. Geom. Topol. \textbf{20} (2020), 1325--1370.

\bibitem[Fio$_3$]{F3} E. Fioravanti, \emph{The Tits alternative for finite-rank median spaces}, Enseign. Math. \textbf{64} (2018), 89--126.

\bibitem[Fio$_4$]{F4} E. Fioravanti, \emph{Superrigidity of actions on finite-rank median spaces}, Adv. Math. \textbf{352} (2019), 1206--1252.

\bibitem[GLP]{GLP} D.~Gaboriau, G.~Levitt, F.~ Paulin, \emph{Pseudogroups of isometries of R and Rips' theorem on free actions on R-trees}, Israel J. Math. \textbf{87} (1994), no. 1-3, 403--428. 

\bibitem[Gen17]{Gen17} A.~Genevois, \emph{Cubical-like geometry of quasi-median graphs and applications to geometric group theory}, preprint \textsc{https://arxiv.org/1712.01618}.

\bibitem[Ger97]{Gerasimov:semisplittings}
V.~N. Gerasimov, \emph{Semi-splittings of groups and actions on cubings},
  Algebra, geometry, analysis and mathematical physics (Russian) (Novosibirsk,
  1996), Izdat. Ross. Akad. Nauk Sib. Otd. Inst. Mat., Novosibirsk \textbf{190} (1997), 91--109.

\bibitem[Ger98]{Gerasimov:fixedpoint}
V.N. Gerasimov, \emph{Fixed-point-free actions on cubings}, Siberian Adv. Math.
  \textbf{8} (1998), no.~3, 36--58.

  \bibitem[Gol]{Gol} William M. Goldman. Complex hyperbolic geometry. Oxford Mathematical Monographs. Clarendon
Press, Oxford University Press, New York, 1999.
  
\bibitem[Gro]{Gro} M. Gromov. Asymptotic Invariants of Infinite
Groups. Geometric Group Theory(vol. 2), G. A. Niblo, M. A. Roller
(eds), Proc. of the Symposium held in Sussex, LMS Lecture Notes
Series 181, Cambridge University Press 1991.

  \bibitem[Grov]{Groves}D. Groves. {\it Limit groups for relatively hyperbolic groups, I: The basic tools}, Algebraic and Geometric Topology, \textbf{9} (2009) 1423-1466.
  
\bibitem[Gui05]{Guirardel:trees}
V.~Guirardel, \emph{{Actions of finitely generated groups on R-trees}},
  Annales de l'institut Fourier, \textbf{58} (2008), no.~1, 159--211.
  
  \bibitem[Hae]{Hae}Th.~Haettel, \emph{Higher rank lattices are not coarse median},
Algebraic \& Geometric Topology \textbf{16} (2016) 2895--2910.

\bibitem[HP21]{HP21} M.~Hagen and H.~Petyt, \emph{Projection complexes and quasimedian maps}, preprint \textsc{https://arxiv.org/2108.13232}.

\bibitem[Ha]{HagP} F. Haglund, \emph{Isometries of CAT(0) cube complexes are semi-simple,} preprint \textsc{https://arxiv.org/abs/0705.3386}.

\bibitem[HP98]{HaPa}
F.~Haglund and F.~Paulin, \emph{{Simplicit\'e de groupes d'automorphismes
  d'espaces \`a courbure n\'egative}}, The Epstein Birthday Schrift (C.~Rourke,
  I.~Rivin, and C.~Series, eds.), Geometry and Topology Monographs, vol.~1,
  International Press, 1998, pp.~181--248.
  
\bibitem[HO21]{HO21} 
J.~Huang and D.~Osajda, \emph{Helly meets {G}arside and {A}rtin}, Invent. Math. \textbf{225} (2021), 395--426.

\bibitem[Isb80]{Isbell}
J.~R. Isbell, \emph{Median {A}lgebra}, Trans. Amer. Math. Soc. \textbf{260} (1980), 319--362.

\bibitem[Kar]{Kar} A. Kar, \emph{Discrete Groups and CAT(0) Asymptotic Cones}, PhD Thesis 2008, Ohio State University, Mathematics.

\bibitem[Kap]{Ka} M.~Kapovich, \emph{Noncoherence of arithmetic hyperbolic lattices}, Geom. \& Topology \textbf{17} (2013), 39--71. 

\bibitem[Me]{Me}M. L. Messaci \emph{Isometric actions on locally compact finite rank median spaces}, preprint 2023, https://arxiv.org/abs/2309.10760

\bibitem[LM]{LiMillson}J.-S. Li, J. J. Millson, \emph{On the first Betti number of a hyperbolic manifold with an arithmetic fundamental group}, Duke Math. J. \textbf{71} (1993) 365--401.

\bibitem[MP14]{MonodPy14} N.~Monod and P.~Py, \emph{An exotic deformation of the hyperbolic space}, Amer. J. Math. \textbf{136} (2014), 1249--1299.

\bibitem[MMR]{MMR}
F.R. McMorris, H.M. Mulder, and F.S. Roberts, \emph{{The median procedure on
  median graphs}}, Report 9413/B, Econometric Institute, Erasmus University
  Rotterdam.
  
\bibitem[NWZ]{NWZ}
G.~Niblo, N.~ Wright and J.~Zhang, \emph{Coarse median algebras: The intrinsic geometry of coarse median spaces and their intervals}, Selecta Mathematica (N.S.) {\bf 27} (2021), Paper No. 20, 50.

\bibitem[Nic]{NicaMaster}
B.~Nica, \emph{Group actions on median spaces}, Master Thesis.

\bibitem[Pe]{Petyt} H.~Petyt, \emph{Mapping Class Groups are quasicubical}, preprint \textsc{https://arxiv.org/2112.10681}.

\bibitem[PSZ]{PSZ} H.~Petyt, D. Spiriano, A. Zalloum. \emph{Hyperbolic models for CAT(0) spaces}. Adv. Math., Vol. 450 (2024), 1–66.

\bibitem[Rol98]{Roller:median}
M.~Roller, \emph{{Poc Sets, Median Algebras and Group Actions}}, preprint \texttt{https://arxiv.org/abs/1607.07747}.
  
\bibitem[Sel97]{Sela:acces}
Z.~Sela, \emph{{Acylindrical accessibility for groups}}, Invent. Math.
  \textbf{129} (1997), no.~3, 527--565.

\bibitem[Sho54a]{Sholander1}
M.~Sholander, \emph{Medians and betweenness}, Proc. Amer. Math. Soc. \textbf{5}
  (1954), 801--807.

\bibitem[Sho54b]{Sholander2}
\bysame, \emph{Medians, lattices and trees}, Proc. Amer. Math. Soc. \textbf{5}
  (1954), 808--812.
  
\bibitem[VS]{VS} E. B. Vinberg, O. V. Shvartsman, \emph{Discrete groups of motions of spaces of constant curvature}, from: ``Geometry, II'', Encyclopaedia Math. Sci. 29, Springer, Berlin (1993), 139--248.

\bibitem[WW75]{WellsWilliams:Embeddings}
J.H. Wells and L.R. Williams, \emph{Embeddings and extensions in analysis},
  Springer-Verlag, New York-Heidelberg, 1975.
  
\bibitem[Wil08]{Wildstrom}
J.~Wildstrom, \emph{Cost thresholds for dynamic resource location}, Discrete
  Applied Mathematics \textbf{156} (2008), 1846--1855.
\bibitem[Yu]{Yu}
G. Yu, \emph{Hyperbolic groups admit proper affine isometric actions on $\ell^p$-spaces}, Geom. Funct. Anal. 15 (2005), 1144–1151.
\end{thebibliography}


\newcommand{\etalchar}[1]{$^{#1}$}
\providecommand{\bysame}{\leavevmode\hbox to3em{\hrulefill}\thinspace}
\providecommand{\MR}{\relax\ifhmode\unskip\space\fi MR }
\providecommand{\MRhref}[2]{%
  \href{http://www.ams.org/mathscinet-getitem?mr=#1}{#2}
} \providecommand{\href}[2]{#2}

\end{document}